\documentclass[a4paper,11pt]{article}
\usepackage{a4wide}
\usepackage{amssymb}
\usepackage{amsmath}
\usepackage{enumerate}
\usepackage{amsthm}
\usepackage{amsfonts}

\newtheorem{theorem}{Theorem}[section]
\newtheorem{definition}{Definition}[section]

\newtheorem{proposition}{Proposition}[section]

\newtheorem{lemma}{Lemma}[section]
\newtheorem{remark}{Remark}

\newcommand{\R}{\mathbb{R}}
\newcommand{\N}{\mathbb{N}}
\newcommand{\DD}{\mathcal{D}}
\newcommand{\FF}{\mathcal{F}}
\newcommand{\e}{\varepsilon}

\DeclareMathOperator{\sign}{sgn}

\begin{document}

\title{Vanishing viscosity limit of a conservation law regularised by a Riesz-Feller operator}



\author{X. Diez-Izagirre, C.~M. Cuesta \\
  University of the Basque Country UPV/EHU,\\ Department of Mathematics, \\Barrio Sarriena S/N, 48940 Leioa, Spain\\
xuban.diez@ehu.eus, carlotamaria.cuesta@ehu.eus}

\date{}



\maketitle

\begin{abstract}
We study a nonlocal regularisation of a scalar conservation law given by a fractional derivative of order between one and two. The nonlocal operator is of Riesz-Feller type with skewness two minus its order. This equation describes the internal structure of hydraulic jumps in a shallow water model. The main purpose of the paper is the study of the vanishing viscosity limit of the Cauchy problem for this equation. First, we study the properties of the solution of the regularised problem and then we show that solutions converge to the entropy 
solution of the scalar conservation law in this limit in $C([0,T];L^1_{loc}(\R))$ for initial data in $L^\infty(\R)$, and in $C([0,T];L^1(\R))$ for initial data in $ L^\infty(\R)\cap BV(\R)$. In order to prove these results we 
use weak entropy inequalities and the double scale technique of Kruzhkov. Such techniques also allow to show the $L^1(\R)$ contraction of the 
regularised problem. For completeness, we study the behaviour in the tail of travelling wave solutions for genuinely nonlinear fluxes. These waves converge to shock waves in the vanishing viscosity limit, but decay algebraically as $x-ct \to \infty$, rather than exponentially, the latter being a behaviour that they exhibit as $x-ct \to- \infty$, however. Finally, we generalise the results concerning the vanishing viscosity limit to Riesz-Feller operators.


\end{abstract}

\section{Introduction}\label{sec:intro}
In this paper we study the following one-dimensional nonlocal problem
\begin{equation}
\begin{cases}
\partial_t u + \partial_x f(u) = \partial_x \DD^{\alpha}[u], & t > 0, \ x \in \R, \\
u(0,x) = u_0(x), & x \in \R,
\end{cases}
\label{Problem1}
\end{equation}
where $f \in C^{\infty}(\mathbb{R})$. Here the notation $\DD^{\alpha}$ stands for the nonlocal operator, acting only on $x$,
\begin{equation}\label{frac:der}
\DD^{\alpha}[g](x) = d_\alpha \int_{-\infty}^{x}{\frac{g'(z)}{(x-z)^\alpha}dz}, \ \ \ 0<\alpha<1, \quad d_\alpha:=\frac{1}{\Gamma(1-\alpha)},
\end{equation}
which corresponds to a Riesz-Feller differential operator (see Section~\ref{sec:rieszfeller}),
and it can be seen as a left-sided Caputo type fractional derivative of order 
$\alpha$ integrated from $-\infty$. It also has an equivalent formulation, that corresponds to the right Weyl-Marchaud derivative (see \cite{Weyl}, \cite{Marchaud}), as we shall remark in Section~\ref{sec:max:glob}.

The equation in (\ref{Problem1}) is used as a model for the far-field behavior of uni-directional viscoelastic waves \cite{SK1985}, and has been also derived as a model for the internal structure of hydraulic jumps in near-critical single-layer flows \cite{Kluwick2010}. The aim of this paper is to complement the recent studies related to these models, namely \cite{AHS} and \cite{AHS2} (see also \cite{ACH} for a related model) where the existence and stability of travelling waves is studied, and most recently \cite{AU2018}, where stability of travelling waves is stablished with decay rates.

In this manuscript we focus on the vanishing viscosity limit associated to (\ref{Problem1}). 
For that reason we shall introduce a control 
parameter in front of the non-local term, $\varepsilon>0$, and let 
$\varepsilon\to 0^+$ in the one-parameter family of problems 
\begin{equation}
\begin{cases}
\partial_t u^\varepsilon(t,x) + \partial_x f(u^\varepsilon) = \varepsilon \partial_x \DD^{\alpha}[u^\varepsilon], & t > 0, \ x \in \R, \\
u^\varepsilon(0,x) = u_0(x), & x \in \R,
\end{cases}
\label{Problem2}
\end{equation}
with the same initial condition.

For $u_0(x) \in L^\infty(\R)$, we prove that the family of solutions converges to the unique entropy solution of the initial value problem for the corresponding conservation law: 
\begin{equation}
\begin{cases}
\partial_t u(t,x) + \partial_x f(u) = 0, & t > 0, \ x \in \R, \\
u(0,x) = u_0(x), & x \in \R.
\end{cases}
\label{Problem3}
\end{equation}
We recall that, formally, the entropy solution of (\ref{Problem3}) satisfies the entropy inequality
\begin{equation}\label{entropy0}
\partial_t \eta(u(t,x)) + \partial_x q(u(t,x)) \leq 0,
\end{equation}
for every $\eta \in C^2(\R)$ convex and $q$ such that $q'(\xi)=\eta'(\xi)f'(\xi)$ (see e.g. \cite{Serre}).

The paper is organised as follows. Before we prove the vanishing viscosity limit, we give in Section~\ref{sec:pre} some preliminary results. We first prove existence, uniqueness and regularity of 
mild solutions that are related to (\ref{Problem1}) by 
Duhamel's principle. The existence is global in time, this is done by proving 
a maximum principle. The proof of this principle requires an equivalent formulation of the nonlocal 
operator, which is only valid for sufficiently regular functions. 
Before we can conclude this, we give some necessary results, 
these include the precise definition of mild solution and some properties 
of the semigroup kernel generated by the linear part of (\ref{Problem1}) 
(i.e. the initial value problem for the same equation without the nonlinear 
term). Many of the above listed results follow from those in \cite{Droniou1}, except for some proofs where it is more convenient to use Fourier transform
and its properties instead of splitting the nonlocal operator in a convenient way; our nonlocal operator has a complex \textit{Fourier} symbol 
(as it is in general for Riesz-Feller operators), and the operator treated in \cite{Droniou1} has a real one (as for the fractional Laplacian, for instance).

Then in Section~\ref{sec:entropy:l1} we prove a weak entropy inequality and the $L^1$ contraction property for (\ref{Problem1}). In Section~\ref{sec:limit} we prove the vanishing viscosity limit for this problem. Both proofs are based on the doubling variable technique of Kruzhkov \cite{Kruzhkov}. The first is readily adapted from the work \cite{Cifani}, the main difference being that our pseudo-differential operator is not symmetric. The vanishing viscosity limit
follows the suggestions given in \cite{Droniou3} for symmetric operators; although the authors do not prove the limit, they do give indications of the steps to be followed. The limit is proved in \cite{Droniou2} for operators with real Fourier symbol, but our proof differs from this one, in that it does not require a splitting in the time evolution of the problem. Indeed, just from the entropy inequalities and the $L^1$ contraction properly, we obtain similar results. Namely, convergence holds in $C([0,T]; L^1_{loc}(\R))$ for $u_0 \in L^\infty(\R)$. And if, in addition, $u_0 \in L^\infty(\R)\cap BV(\R)$, we obtain convergence in $C([0,T]; L^1(\R))$ with the following estimate:
\begin{equation*}
\|u^\e(t,\cdot)- u(t,\cdot )\|_{1} \leq C\left(\e \, t \right)^{\frac{1}{\alpha+1}}|u_0|_{BV}.
\end{equation*}
Here we include the dependency on $\e$ in the decay rate, as this is valid in the limit $\e\to 0^+$.

In Section~\ref{sec:TW} we analyse the decay to far-field values of travelling wave solutions associated to (\ref{Problem1}). These solutions converge pointwise to a shock wave as $\e\to 0^+$. The existence of these solutions is shown in \cite{AHS}, here we complete the analysis by showing that as $x-ct \to\infty$ the decay to the constant right value is algebraic. 

For completeness, in Section~\ref{sec:rieszfeller} we consider the vanishing viscosity limit associated to the scalar conservation law with more general regularising viscosity:
\begin{equation}
 \begin{cases}
\partial_t u + \partial_x f(u) = D^\beta_\gamma[u], & t > 0, \ x \in \R, \\
u(0,x) = u_0(x) & x \in \R,
\end{cases}
\label{Problem:general} 
\end{equation}
where $\beta\in(1,2]$ and $|\gamma|\leq \min\{\beta,2-\beta\}$, and $D^\beta_\gamma$ is a Riesz-Feller operator of order $\beta$ and skewness $\gamma$, that for these parameters is defined by means of
a Fourier multiplier operator (see e.g. \cite{MLP})
\begin{equation}\label{symbol1:general}
\FF(D^{\beta}_\gamma[ u])(\xi)=\psi_\gamma^\beta(\xi) \, \FF(u)(\xi),
\end{equation}
where the symbol reads
\begin{equation}\label{symbol2:general}
\psi_\gamma^\beta(\xi)=-|\xi|^{\beta}\, \exp\left[ i \, \sign(\xi)\gamma\frac{\pi}{2}\right] \quad \mbox{for} \quad 1<\beta\leq 2, \ |\gamma|\leq \min\{\beta,2-\beta\}. 
\end{equation}
We observe that (\ref{frac:der}) is of this form with $\beta=1+\alpha$ and $\gamma=1-\alpha$ (see (\ref{f:symbol}) below).

Once we have analysed problem (\ref{Problem1}), the generalisation to (\ref{Problem:general}) requires a minimal effort if we use the integral representation of the Riesz-Feller operators (see Proposition~2.3 of \cite{AK2} and e.g. \cite{Sato}). This representation is closely related to (\ref{frac:der}) and its adjoint operator, as we shall see, it is a linear combination of both for smooth enough functions. We have chosen to focus first on the specific example given in (\ref{Problem1}) for two reasons: first, because this is the example that has come to us from applications, and second because the proofs, although analogous, involve shorter formulas.

We conclude this introduction by mentioning that in recent years hyperbolic problems with nonlocal regularizations that generalise the fractional Laplacian has been extensively studied. These include linear (see \cite{Droniou2}) as well as nonlinear regularizations, regularizations of order lower than or equal to one (see e.g. \cite{Droniou3} and \cite{Alibaud2007}) as well as degenerate diffusion ones (see e.g. \cite{EJ2014} and references therein). Although the techniques we use are similar and based on doubling variables, the type of operator, we are interested in, is not included in the classes analysed in this literature. Moreover, in contrast to the results in \cite{Droniou1} and \cite{Droniou2}, we do not require a splitting in the time evolution that alternates solving the conservation law and then the diffusion equation; notice that this method explicitly chooses the entropy solution in the time intervals where the former is solved.

\section{Preliminary results}\label{sec:pre}
In this section, we define an equivalent formulation of (\ref{Problem1}), the mild formulation that is based on Duhamel's principle. 
Then we give some properties of the corresponding kernel (or semigroup) associated to this formulation. Many of the steps in the proofs 
that follow are similar to those in \cite{Droniou1} and in \cite{AHS}. We have proved some properties of the kernel differently and
 we report on them. Finally, we give the existence and uniqueness results for the mild initial value problem and give a global existence 
result. The last step requires to prove a  maximum principle which is based on an equivalent representation of our non-local operator
 applied to smooth enough functions.

 Before we continue let us introduce some notation and give some properties of (\ref{frac:der}) and its derivative with respect to $x$.
 
Notice, that here and throughout we use the notation
$\|\cdot\|_{1}$ for the norm of $L^1(\R)$, $\|\cdot\|_{\infty}$ for the norm of $L^\infty(\R)$, and for functions of bounded variation in $x$, we have
\[
|u|_{BV}:=\sup\left\{\int_\R u(x)\,\phi'(x)\,dx: \ \phi\in C_c^1(\R), \  \|\phi\|_\infty\leq 1 \right\}.
\]
We recall that if $u\in W^{1,1}(\R)$ then $|u|_{BV}=\|u'\|_1$, and if also $u\in C^1(\R)$, then $\int_\R |u(x+h)-u(x)|dx\leq |h| |u|_{BV}$.

We use the following definition and notation for the Fourier transform:
\begin{equation*}
\FF(u(x))(\xi) = \hat{u}(\xi)= \frac{1}{\sqrt{2\pi}} \int_{\R}{u(x) e^{-i\xi x} \ dx}.
\end{equation*}
In order to compute the Fourier transform of $\partial_x \DD^{\alpha}[u]$, we rewrite it as a convolution, 
\begin{equation}\label{op:conv}
\DD^{\alpha}[u](x)= d_\alpha \left( \theta(\cdot) \ (\cdot)^{-\alpha} * u'\right)(x)
\end{equation}
where $\theta$ is the Heaviside function. Then (see e.g. \cite{AA})
\begin{equation}\label{f:symbol}
\FF(\partial_x\DD^{\alpha}[u])(\xi)=(i\xi)^{\alpha+1} \ \FF(u)(\xi).
\end{equation}

It is not hard to see, splitting the integral and using integration by parts in one of the resulting integrals, that the operator (\ref{frac:der}) is bounded from $C_b^1$ to $C_b$ and from $H^{m+\alpha}$ to $H^m$. This type of argument will be used in subsequent proofs to get more precise estimates.

\subsection{Mild solutions}
Let us define mild solutions for (\ref{Problem1}) using Duhamel's principle and Fourier transform.

In view of (\ref{f:symbol}) 
we define the kernel 
\begin{equation}\label{kernel}
K(t,x)=\FF^{-1}\left(e^{(i\xi)^{\alpha+1}t}\right)(x)\quad  \forall t>0, \quad x\in \R
\end{equation}
and formally obtain, by Duhamel's principle, the solution to (\ref{Problem1})
\begin{equation*}
u(t,x)= K(t,\cdot)*u_0(x) - \int_{0}^{t}K(t-s, \cdot)*\partial_xf(u(s,\cdot))(x)\, ds.
\end{equation*}
For convenience, we write the derivative of $f(u)$ in $K$ in the convolution, and we arrive at the following definition of mild solution:
\begin{definition}[Mild solution]\label{solution}
Given $T\in (0,+\infty]$ and $u_0\in L^{\infty}(\R)$, we say that a \textit{mild solution} of (\ref{Problem1}) on $(0,T)\times \R$ is a function 
$u\in L^\infty((0,T)\times\R)$ which satisfies 
\begin{equation}\label{solutioneq}
u(t,x)=K(t,\cdot)*u_0(x) - \int_{0}^{t}\partial_x K(t-s,\cdot)*f(u(s,\cdot))(x)\, ds \quad \mbox{a.e. in} \quad (t,x)\in(0,T)\times\R.
\end{equation}
\end{definition}

Some properties of $K$ have been already proved in \cite{AHS}, here we add two more properties related to its derivatives with respect to $t$.

\begin{proposition}[Properties of the kernel $K$]\label{K:prop}
For $0<\alpha<1$, the kernel $K$ given in (\ref{kernel}) is non-negative. Additionally, $K$ satisfies the properties:
\begin{enumerate}[(i)]
\item Self-similarity: For all $t>0$ and $x\in \R$, $$K(t,x) = \frac{1}{t^{\frac{1}{1+\alpha}}} K\left(1, \frac{x}{t^{\frac{1}{1+\alpha}}}\right).$$

\item Mass conservation: For all $t>0$, $\|K(t, \cdot)\|_{1} = 1.$

\item Semigroup property: $\forall a,b \in (0,\infty)$, 
\[
K(a,\cdot)*K(b,\cdot)=K(a+b,\cdot)
\]
 and 
\[
K(a,\cdot)*\partial_xK(b,\cdot)
=\partial_x\left(K(a+b,\cdot)\right).
\]

\item Space regularity: $K(t,x)\in C^\infty((0,\infty)\times\R)$ and for all $m\geq 0$ there exists a $B_m>0$ such that
\begin{equation*}
 |\partial_x^m K(t,x)| \leq \frac{1}{t^{\frac{1+m}{1+\alpha}}}\frac{B_m}{(1+t^{\frac{-(m+2)}{1+\alpha}}|x|^{m+2})}\quad \mbox{for all}\quad (t,x)\in (0,\infty)\times\R.
\end{equation*}
 In particular, there exist $C_m>0$ such that for all $m\geq 1$ and $t>0$: 
\[
\|\partial_x^m K(t,\cdot)\|_1 = \frac{C_m}{t^{\frac{m}{1+\alpha}}}.
\]

\item Time regularity: $K(t,x) \in C^\infty((0,\infty)\times\R)$ and for all $m \geq 0$
 there exist constants $K_m>0$ such that, for all $(t,x) \in (0,\infty) \times \R$,
\[
 |\partial_t^m K(t,x)| \leq \frac{1}{t^{\frac{m(\alpha+1)+1}{\alpha+1}}} \frac{K_m}{1 + \left(\frac{|x|}{t^{1/(\alpha+1)}}\right)^{2}}.
\]
As a result there exists a $D_m>0$ such that for all $t>0$:
\[
\|\partial_t^m K(t,\cdot)\|_{1} = \frac{D_m}{t^m}.
\]
\end{enumerate}
\end{proposition}

\begin{proof}
  The non-negativity of the kernel follows from the fact that the kernel is the scaled probability measure of a L\'{e}vy strictly $\alpha+1$-stable distribution (\cite{Sato}), hence it has to be a non-negative function.

The proofs of (i), (ii) and (iii) can be done as in \cite{Droniou1} or in \cite{AHS}. 
Let us prove (iv). By (i) and the change of variables
$y=\frac{x}{t^{\frac{1}{1+\alpha}}}$, we get, for all $m\geq 1$ and all $t>0$, that
\begin{equation}\label{simple:bdd}
\begin{split}
  |\partial_x^m K(t,x)|=\frac{1}{t^{\frac{m+1}{1+\alpha}}}|\partial_y^m K(1,y)| =\frac{1}{\sqrt{2\pi}\,t^{\frac{m+1}{1+\alpha}}}\left| \int_{\R}{(i\xi)^m  e^{(i\xi)^{\alpha+1}} \, e^{I\xi} \ d\xi}\right|\\
  \leq 
\frac{\sqrt{2}}{\sqrt{\pi}\,t^{\frac{m+1}{1+\alpha}}}\int_0^\infty |\xi|^me^{-|\xi|^{\alpha+1}\sin\left(\frac{\alpha \pi}{2}\right)} d\xi= \frac{ C_{\alpha,m}}{t^{\frac{m+1}{1+\alpha}}}<\infty 
\end{split}
\end{equation}
with
\[
C_{\alpha,m} = \frac{\sqrt{2}\,\Gamma\left(\frac{1+m}{1+\alpha}\right)}{\sqrt{\pi}(\alpha+1)\,\sin^{\frac{m+1}{\alpha+1}}\left(\frac{\alpha \pi}{2}\right)},
\]
where we have used the property $\partial_\xi^m\FF(\varphi(x))(\xi)= \FF((i x)^m\varphi(x))(\xi)$ for $m\in \N$ and 
the change of variables $z=\sin\left(\frac{\alpha\pi}{2}\right) \xi^{\alpha+1}$.

Let us finally show that the maximal decay of this $|\partial_x^m K(t,x)|$ is slower than or equal to $O\left((|x|/t^{1/(\alpha+1)})^{-(m+2)}\right)$ as 
$|x|/t^{1/(\alpha+1)}\to\infty$. We do this, using again (i) and the self-similar variable $y$. We observe that $K(1,y)=O\left(\frac{1}{y^2}\right)$
as $y \to\infty$. Indeed, applying integration by parts twice, we have
\begin{equation*}
K(1,y) 
= \frac{1}{\sqrt{2\pi}} \frac{\alpha+1}{y^2} \int_{\R}{\left(\alpha (i\xi)^{\alpha-1} + (\alpha+1) (i\xi)^{2\alpha}\right) 
\ e^{(i\xi)^{\alpha+1}} e^{i y \xi} \ d\xi}.
\end{equation*}
Then, arguing as above, there exist constants $C_1$, $C_2$,  such that
\[
|K(1,y)|\leq\frac{1}{y^2} \left( C_1\,\Gamma\left( \frac{\alpha}{\alpha+1} \right) + C_2 \, \Gamma\left( \frac{2\alpha+1}{\alpha+1} \right) \right).
\]
We now apply induction. Observe that integration by parts of $\partial_y^m K(1,y)$, gives 
\[
|\partial_y^m K(1,y)|\leq \frac{C}{|y|^{m+2}} + \frac{\alpha+1}{\sqrt{2\pi}}\frac{1}{|y|}\left|\int_\R (i\xi)^{m+\alpha} e^{(i\xi)^{\alpha+1}}e^{i\xi}d\xi\right|,
\]
where we have applied the induction hypothesis to the first term. The second term can be integrated by parts $m+1$ times and, as before, changing variables and using the definition of the Gamma function), to get for some constants $A_{k,l}>0$,
\[
|\partial_y^m K(1,y)|\leq \frac{1}{|y|^{m+2}} 
\left(C+ \sum_{\substack{k,l=1\\ k+l=m+1}}^{m+1} A_{k,l}\,\Gamma\left(\frac{(m-k)+(1+l)\alpha}{\alpha+1}\right) \right).
\]
This and (\ref{simple:bdd}), by changing to the original variables, imply (iv). 
In addition, we get that $\partial_x^m K(t,x)$ are continuous on $(0,\infty)\times\R$ for all $m\geq 0$ by continuity under the integral sign.

Let us finally prove (v). We first observe that, using (i),
\[
\partial_t K(t,x)=-\frac{1}{\alpha+1}\frac{1}{t}\left(K(t,x)+x\,\partial_x K(t,x)\right)
\]
and, by induction, we have that there exist positive constants such that
\begin{equation*}
\partial_t^mK(t,x) = \frac{1}{t^m} \sum_{j=0}^{m}{C_j  x^j \ \partial_x^j K(t,x)}, \ \ C_j\in \R \ \text{ for all } 0 \leq j \leq m. 
\end{equation*}
Now, we apply (iii), then for all $m\geq 0$ 
\begin{equation*}
|\partial_t^mK(t,x)| \leq \frac{1}{t^{m+\frac{1}{\alpha+1}}}
 \sum_{j=0}^{m}  \frac{ C_j\left(\frac{|x|}{t^{1/(\alpha+1)}}\right)^j}{\left(1 + \left(\frac{|x|}{t^{1/(\alpha+1)}}\right)^{j+2}\right)} \leq \frac{K_m}{t^{m+\frac{1}{\alpha+1}} \left(1 + \left(\frac{|x|}{t^{1/(\alpha+1)}}\right)^{2}\right)},
\end{equation*} 
for all $(t,x)\in (0,\infty)\times \R$.

Again that $\partial_t^m K(t,x)$ are continuous on $(0,\infty)\times\R$ 
and for all $m\geq 0$ follows by continuity under the integral sign. 
And the $L^1$ norm property is proved using the last inequality.
\end{proof}

The following proposition shows that all the terms in Definition~\ref{solution} are well-defined if $u\in L^\infty((0,T)\times\R)$:
\begin{proposition}\label{welldefined}
Let $T>0$, $u_0\in L^\infty(\R)$ and $T>0$ and $v \in L^\infty((0,T)\times\R)$, then,
\[
u(t,x):=K(t,\cdot)*u_0(x) -\int_{0}^{t}\partial_x K(t-s,\cdot)*v(s,\cdot)(x)\, ds\in C_b((0,T)\times\R).
\]
Moreover, for all $t_0\in (0,T)$, $x\in \R$ and $t\in (0,T-t_0)$, 
\begin{equation*}
u(t_0+t,x)=K(t,\cdot)*u(t_0,\cdot)(x) - \int_{0}^{t}{\partial_x K(t-s,\cdot)*v(t_0+s,\cdot)(x)\ ds}.
\end{equation*}
\end{proposition}
The proof can be adapted easily to our case from that given in \cite{Droniou1}.

We shall need the following
\begin{lemma}\label{continuityproperties}
  Let $T>0$ and $(t_0,x_0) \in (0,T) \times \R $. If $v \in C_b((0,T)\times \R)$, then
\begin{enumerate}[(i)]
\item For all $s_0>0$, $\lim_{(s,t,x)\to (s_0,t_0,x_0)}K(s,\cdot)*v(t,\cdot)(x)= K(s_0,\cdot)*v(t_0,\cdot)(x_0)$.
\item $\lim_{(s,t,x)\to (0,t_0,x_0)}K(s,\cdot)*v(t,\cdot)(x) = v(t_0,x_0)$.
\end{enumerate}
\end{lemma}
The proof of this lemma  uses the dominated convergence theorem and Proposition~\ref{K:prop}.

We can now show that the operator $K(t,\cdot)\ast u_0$ is a classical solution of the linear part of problem (\ref{Problem1}) :

\begin{proposition}\label{timeregularity1}
If $u_0 \in C_b(\R)$, let $U(t,x):=(K(t,\cdot)*u_0)(x)$ for all $(t,x)\in(0,\infty)\times\R$, then $U \in C^\infty((0,\infty)\times\R)$ and satisfies
\begin{equation}\label{lin:eq}
\partial_t U = \partial_x \DD^\alpha[U]
\end{equation}
with $ \lim_{t\to 0^+}U(0,x)=u_0(x)$ for all $x\in\R$.

If $u_0 \in L^\infty(\R)$, then also $U\in C^\infty((0,\infty)\times\R)$ satisfies (\ref{lin:eq}), but we can only assure that $U(t,\cdot)\to u_0$ as $t\to 0^+$ in $L^1_{loc}(\R)$.
\end{proposition}
For a proof we refer to \cite{AK2}, where the result is proved for general Riesz-Feller operators. The last statement about convergence to the initial condition follows by classical results about smoothing by convolution (see e.g. \cite{Tartar}).

\subsection{Existence and Regularity results}\label{sec:exis}
The proofs of local existence and uniqueness of mild solutions of this section are based on those given in \cite{Droniou1} and use 
Proposition~\ref{K:prop}, we shall not give all the details here. 

\begin{proposition}[Existence, uniqueness and space regularity]\label{local:existence}
Let $u_0 \in L^\infty(\R)$ and $f \in C^\infty(\R)$, and let $\|u_0\|_\infty=R_0$. Then, there exists $T>0$, only depending on $R_0$, 
such that, there exists a unique $u \in C_b^\infty((0,T)\times \R)$, that satisfies Definition~\ref{solution}.

Moreover, for all $m\in\N\cup \{0\}$ there exists a $C_m>0$ depending on $t_0$ and $T$ such that 
$\|\partial_x^m u\|_{C_b((t_0,T)\times\R)}<C_m$, and where $t_0=0$ if $m=0$. 
Also, for every $m\in\N$, $t_0 \in (0,T)$ and $t \in (0, T-t_0)$, it holds
\[
\partial_x^m u(t_0+t, \cdot) = K(t, \cdot)*\partial_x^m u(t_0, \cdot) - \int_{0}^{t}{\partial_x K(t-s, \cdot)*\partial_x^m (f(u(t_0+s, \cdot)))\ ds}.
\]
\end{proposition}

\begin{proof}
The proof uses a contraction mapping argument. For a fixed $T >0$, one first defines the following Banach space
\[
E_T=\{v\in C_b((0,T)\times\R): \ \partial_xv \in C((0,T)\times\R) \ \mbox{and} \ t^{\frac{1}{\alpha+1}}\partial_xv \in C_b((0,T)\times\R)\},
\]
endowed with the norm 
\[
\|v\|_{E_T}= \sup_{t\in(0,T)}{\left\{ \|v(t,\cdot)\|_{L^\infty(\R)} + \left\|t^{\frac{1}{\alpha+1}}\partial_x v(t,\cdot) \right\|_{L^\infty(\R)}\right\}},
\]
and, the fix-point map $\Psi_T: C_b((0,T)\times\R) \longrightarrow C_b((0,T)\times\R)$ by means of 
\[
\Psi_T(v)(t,x) = \left(K(t,\cdot)*u_0\right)(x) - \int_{0}^{t} \left( \partial_xK(t-s,\cdot)*f(v(s,\cdot)\right)(x) \,ds.
\]
With the aid of Lemma~\ref{continuityproperties} and Proposition~\ref{timeregularity1}, one can show that $\Psi_T(v)$ belongs to $E_T$ for all $v\in E_T$ and that, in fact, $\Psi_T$ maps $B_T(R)$ into itself for some $R>R_0$,
 where $B_T(R)$ denotes the closed ball in $E_T$ of centre $0$ and radius $R>0$. Finally, one shows that
\begin{equation*}
\left\|\Psi_T(u)-\Psi_T(v)\right\|_{E_T} \leq T^{\frac{\alpha}{\alpha+1}}\,C(R_0)\, \|u - v\|_{E_T} ,\quad u,v\in B_T(R),
\end{equation*}
thus $\Psi_T$ is a contraction in $B_T(R)$ for a small enough $T$. Then, there exists a unique fixed point $u \in B_T(R)$.
This implies in particular, since $\|u\|_{E_T}\leq R$, that $|u(t,x)|\leq R$ for all $(t,x) \in (0,T)\times \R$ and 
\[
\left| \partial_x u(t,x) \right| \leq a^{\frac{-1}{\alpha+1}} R \quad \mbox{for all} \quad (t,x)\in (a,T)\times \R, \quad a\in (0,T).
\]
Observe the last statement for $m=0$ holds from Proposition~\ref{welldefined}. 

The rest of the proof can be done by induction. Indeed, observe that differentiation of (\ref{solutioneq}) gives a fix-point map 
of the form
\[
v \longrightarrow \partial_x^m  K(t,\cdot)* u_0(x) - \int_{0}^{t} 
\Big( \partial_x K(t-s,\cdot)* \Big( g_m(s,\cdot) + f'(u(s,\cdot)) v(s,\cdot) \Big) \Big)(x)\, ds
\]
where $g_m$ is such that $g_m(t,x) + f'(u(x,t))\partial_x^m u = \partial_x^m(f(u))$. This gives the regularity and bounds on the derivatives of $u$. 
Observe that then, for $t_0>0$ and using Proposition~\ref{welldefined}, one can conclude the last statement also by induction 
and the regularity of $u$.
\end{proof}

In the following proposition we state the temporal regularity of the mild solution.
\begin{proposition}\label{time:reg}
Let $u_0\in L^\infty(\R)$ and $T\in (0,\infty]$. If $u$ satisfies Definition~\ref{solution} in $(0,T)\times\R$, 
then $u$ is infinitely differentiable with respect to $t>0$ and $\partial_t u + \partial_x (f(u)) = \partial_x\DD^\alpha[u]$ on $(0,T)\times \R$. Moreover, $\partial_t^m u\in C^{\infty}_b((0,T)\times \R)$ for $m\in\N$.
\end{proposition}
The proof is analogous to that in \cite{Droniou1}, we do not write it here. The proof uses Lemma~\ref{continuityproperties} and propositions~\ref{timeregularity1} and \ref{local:existence}.

\subsection{A maximum principle and global existence}\label{sec:max:glob}
Here we prove the global existence of solutions of (\ref{Problem1}). Instead of using a splitting method, as in \cite{Droniou1}, we show global existence by a
 maximum principle, as pointed out in \cite{Droniou3}. 

In order to show the maximum principle, we first give an equivalent formulation of (\ref{frac:der}). A related result appears in \cite{AK2}, Proposition~2.3, see also Section~\ref{sec:rieszfeller}.
\begin{lemma}[Equivalent representation of $\partial_x \DD^\alpha$ and $\DD^\alpha$]\label{equiv:repre}
If $\alpha \in (0,1)$, then for all $\varphi \in C^2_b(\R)$ and all $x\in \R$,
\begin{equation}
\partial_x \DD^\alpha[\varphi](x) = d_{\alpha+2} \int_{-\infty}^{0} \frac{\varphi(x+z)-\varphi(x)- \varphi'(x)z}{|z|^{\alpha+2}} dz.
\label{integralformula1}
\end{equation}
Moreover, we can also get this integral formula for the operator $\DD^\alpha$ where $\alpha \in (0,1)$. For all $\varphi \in C_b^1(\R)$ and all $x\in \R$,
\begin{equation}
\DD^\alpha [\varphi] (x) = d_{\alpha+1} \int_{-\infty}^{0}\frac{\varphi(x+z)-\varphi(x)}{|z|^{\alpha+1}} dz.
\label{integralformula2}
\end{equation}
\end{lemma}
\begin{remark}
  We observe that the representation (\ref{integralformula2}) corresponds to minus the Weyl-Marchaud right derivative of order $\alpha$, after the change of variables $z\to -z$, see \cite{Weyl} and \cite{Marchaud}.  
 \end{remark}
\begin{proof}
First we observe that the assumption on $\varphi$ and that $\alpha\in(0,1)$ imply that the expressions on the left-hand side 
of (\ref{integralformula1}) and of (\ref{integralformula2}) are well-defined.

We can now manipulate these integrals. In order to obtain (\ref{integralformula1}), we apply the Fundamental Theorem of Calculus twice and 
interchanging a derivative with the integrals:
\begin{align*}
\int_{-\infty}^{0} & \frac{ \varphi(x+z)-\varphi(x)- \varphi'(x) z}{ |z|^{\alpha+2} }\, dz = 
\int_{-\infty}^{0} \frac{ \int_{0}^{z} \varphi'(x+y)\, dy - \varphi'(x)z }{ (-z)^{\alpha+2} } \, dz \\
& \qquad = \partial_x \int_{-\infty}^{0} \int_{0}^{z} \frac{ \varphi(x+y)- \varphi(x) }{ (-z)^{\alpha+2} } \, dy \,dz
= \partial_x \int_{-\infty}^{0} \int_{0}^{z} \int_{0}^{y} \frac{\varphi'(x+r)}{(-z)^{\alpha+2}}\, dr\,dy \,dz\\
& \qquad = \partial_x \int_{-\infty}^{0}{\int_{-\infty}^{y}{\int_{0}^{y}{\frac{\varphi'(x+r)}{(-z)^{\alpha+2}}\ dr}dz}dy} 
= \frac{1}{\alpha+1} \partial_x \int_{-\infty}^{0}{\int_{0}^{y}{\frac{\varphi'(x+r)}{(-y)^{\alpha+1}}\ dr}dy}
\\
& \qquad = \frac{1}{\alpha+1} \partial_x \int^{0}_{-\infty}{\int_{-\infty}^{r}{\frac{\varphi'(x+r)}{(-y)^{\alpha+1}}\ dy}dr}
= \frac{1}{(\alpha+1)\alpha} \partial_x \int_{-\infty}^{0}{\frac{\varphi'(x+r)}{(-r)^{\alpha}}\ dr}.
\end{align*}
Observe that, by the properties of the Gamma function, $\alpha(\alpha+1)d_\alpha=d_{\alpha+2}$, then (\ref{integralformula1}) follows. Applying similar manipulations one obtains (\ref{integralformula2}) (for this case $-\alpha d_\alpha=d_{\alpha+1}<0$). 
\end{proof}

From (\ref{integralformula1}) and (\ref{integralformula2}), it is obvious that if $\varphi$ attains its global maximum at $x$, then we have
$\partial_x\DD^\alpha(\varphi)(x) \leq 0$ and $\DD^\alpha(\varphi)(x) \leq 0$, with the identity holding if $\varphi$ is constant.
 And from this property we formulate the following lemma, that can be proved as in \cite{Droniou3} using the continuity of the nonlocal operators.

\begin{lemma}[A maximum principle]\label{negativity}
Let $\alpha \in (0,1)$ and $\varphi \in C_b^2(\R)$. If $\{x_n\}_{n = 1}^\infty\subset \R$ such that $\varphi(x_n) \to \sup_{\R}{\varphi(x)}$ as 
$n\to \infty$, then $\lim_{n\to \infty}{\varphi'(x_n)}=0$ and 
\begin{equation}\label{max:prin:1}
{\lim{\sup}_{n\to \infty}{\partial_x \DD^\alpha (\varphi)(x_n)}} \leq 0.
\end{equation}
\end{lemma} 


In the following proposition we give the global existence:
\begin{proposition}[Global existence]\label{global:existence}
Let $\alpha\in(0,1)$, $T>0$. 
If $u\in C_b^2((0,T)\times\R)$ satisfies (\ref{Problem1}), 
then, we have, for all $0<t'<t<T$,
\begin{equation*}
\|u(t,\cdot)\|_{\infty} \leq \|u(t',\cdot)\|_{\infty}.
\end{equation*}
Moreover, if $u$ is a solution as constructed in Theorem~\ref{local:existence}, then
\begin{equation*}
\|u(t,\cdot)\|_{\infty} \leq \|u_0\|_{\infty} \quad \mbox{ for all } \ t\in (0,T)
\end{equation*}
and the solution can be extended globally in time.
\end{proposition}

\begin{proof}
Let $\delta\in (0,T)$. Since, $|\partial_t^2 u|$ is bounded on $(\frac{\delta}{2},T)\times \R$ by some $C_\delta$, we have, 
by performing a Taylor expansion and using the equation, that for all $t\in (\delta,T)$, all $0<\tau<\frac{\delta}{2}$ and all $x\in \R$, 
\begin{equation}
\begin{split}
u(t,x) &\leq u(t-\tau,x) + \tau \partial_t u(t,x) + C_\delta \tau^2 \\
&\leq \sup_{x\in \R}{u(t-\tau,x)} - \tau f'(u(t,x))\partial_xu(t,x) + \tau \partial_x \DD^\alpha [u(t,\cdot)](x) + C_\delta \tau^2.
\label{inequality28}
\end{split}
\end{equation}

For a $t\in(\delta,T)$ let $\{x_n\}_{n\in \N} \in \R$ be a sequence such that 
$u(t,x_n) \to \sup_{x\in \R}{u(t,\cdot)} $ and let $M_t= \sup_{x\in\R} |f'(u(t,x))|$.
Then, by (\ref{inequality28}), we obtain for all $0<\tau<\frac{\delta}{2}$,
\[
u(t,x_n) \leq \sup_{x\in \R}{u(t-\tau,\cdot)} + \tau M_t |\partial_x u(t,x_n)| + \tau \partial_x \DD^\alpha [u(t,\cdot)](x_n) + C_\delta\tau^2,
\]
and Lemma~\ref{negativity} implies, taking the limit $n\to\infty$, that
\[
\sup_{x\in\R}{u(t,x)} \leq \sup_{x\in \R}{u(t-\tau,x)} + C_\delta \tau^2.
\]
This also implies that
\begin{equation}\label{max:decreasing}
\max\{\sup_{x\in\R}{u(t,x)},0\} \leq \max\{\sup_{x\in \R}{u(t-\tau,x)} ,0\}+ C_\delta\tau^2.
\end{equation}
We observe that $\max\{\sup_{x\in\R}{u(t,x)},0\} \in W^{1,\infty}(\delta,T)$, because it is Lipschitz continuous in $(\delta,T)$. Indeed,
\[
\begin{split}
|\max\{\sup_{x\in\R}{u(t,x)},0\} -\max\{\sup_{x\in\R}{u(t',x)},0\} |\leq \max \{ |\sup_x u(t,x)-\sup_x u(t',x) | ,0 \}\\
 \leq \sup_x |u(t,x)-u(t',x)|\leq \sup_{(t,x)} |\partial_t u(t,x)| \,|t-t'|
\end{split}
\] but $|\partial_t u|$ is bounded on $(\delta,T)\times\R$. In particular, (\ref{max:decreasing}) implies that $\max\{0,\sup_x u(t,x)\}$ decreases, so for all $0<t'<t<T$, 
$$
\max\{0,\sup_x u(t,x)\} -\max\{0,\sup_x u(t',x)\} \leq 0.
$$

The same reasoning applied to $v=-u$, which is a solution of
\[
\partial_t v +\partial_x g(v)=\partial_x\DD^\alpha[v] \quad \mbox{with} \quad g(v)=-f(-v), 
\]
gives that for all $0<t'<t<T$, $\max\{0,\sup_x(-u(t,x))\} -\max\{0,\sup_x(-u(t',x))\} \leq 0$, and we conclude the proof of the first statement.

It remains to prove the last statement by taking the limit of $t'\to 0^+$. This follows from Definition~\ref{solution} and 
Proposition~\ref{K:prop}, since for all $t'>0$
\[
\|u(t',\cdot) \|_{\infty} \leq \|u_0 \|_{\infty} + C(t')^{\frac{\alpha}{\alpha+1}} \sup_{t\in(0,t')}\|u(t,\cdot)\|_\infty 
\]
thus $\lim \sup_{t'\to 0^+} \|u(t',\cdot) \|_{\infty} \leq \|u_0 \|_{\infty}$.
\end{proof}

\section{Entropy inequalities and $L^1$ contraction}\label{sec:entropy:l1}
In the limit $\e \to 0^+$ we expect to recover the entropy solution of (\ref{Problem3}), that is the solution that satisfies the 
entropy inequality (\ref{entropy0}). Observe that, formally, multiplying the equation in 
(\ref{Problem2}) by $\eta'(u)$, for some convex $\eta \in C^2(\R)$, we get:
\begin{equation}\label{entropyinequality}
  \partial_t \eta(u^\e(t,x)) + \partial_x q(u^\e(t,x))
  = \e \eta'(u^\e) \partial_x \DD^\alpha[u^\e](x)
\end{equation}
 where $q$ is such that $q'(u)= \eta'(u) f'(u)$. Let us prove a weak version of (\ref{entropyinequality}). We first need the following lemma:
\begin{lemma}\label{integrationbyparts}
Let $\alpha\in(0,1)$, $u\in C^2_b((0,\infty)\times\R)$ and $\varphi\in C^\infty_c(\R)$. Then,
\begin{equation*}
\int_{0}^{\infty}\int_{\R} \varphi(x) \partial_x \DD^\alpha[u(t,\cdot)](x) \, dx\, dt =
 \int_{0}^{\infty}\int_{\R} \partial_x \overline{\DD^\alpha}[\varphi](x) u(t,x) \, dx\,dt,
\end{equation*}
where $\overline{\DD^\alpha}$ is defined by means of
\[
\overline{\DD^\alpha}[g](x)= d_\alpha \int_{x}^{\infty} \frac{g'(z)}{(z-x)^\alpha}\, dz.
\]
Moreover, for $g\in C^2_b(\R)$,
\begin{align*}
\overline{\DD^\alpha}[g](x) & = - d_{\alpha+1} \int_{0}^{\infty} \frac{g(x+z)-g(x)}{|z|^{\alpha+1}}\ dz,
	\\
\partial_x \overline{\DD^\alpha}[g](x) & = d_{\alpha+2} \int_{0}^{\infty} \frac{g(x+z)-g(x)- g'(x)z}{|z|^{\alpha+2}}\ dz.
\end{align*}

\end{lemma}
\begin{remark}
We notice that the equivalent representation of $\overline{\DD^\alpha}[g](x)$ given in this lemma is the left Weyl-Marchaud (\cite{Weyl}, \cite{Marchaud}) fractional derivative of order $\alpha$.
\end{remark}

\begin{proof}
We start with $\partial_x \DD^\alpha$. First we integrate by parts, the we interchange the order of integration, and we integrate by parts a second time, this gives:
\begin{equation}\label{D:Manipu2:bar}
\int_{\R}{\partial_x \DD^\alpha [u(t,\cdot)](x) \varphi(x) \ dx} =  \int_{\R}  u(t,y) \partial_y \overline{\DD^\alpha}[\varphi](y) \, dy 
-\lim_{y\to-\infty}  u(t,y) \overline{\DD^\alpha}[\varphi](y)
\end{equation}
(observe that the first boundary term vanishes trivially). Let us show that the last term vanishes. Since $u\in C_b^2$, it is enough to show that 
$\lim_{y\to-\infty} \overline{\DD^\alpha}[\varphi](y)=0$ for all $\varphi\in C^\infty_c(\R)$. We observe that, for any $r>0$, we can write
\begin{equation}\label{Dbar:split}
\lim_{y\to -\infty}{\overline{\DD^\alpha}[\varphi](y)} = 
\lim_{y\to -\infty}\int_{0}^{r}{\frac{\varphi'(z+y)}{z^\alpha}}\ dz + \lim_{y\to -\infty}\int_{r}^{\infty}{\frac{\varphi'(z+y)}{z^\alpha}}\ dz,
\end{equation}
and the first term vanishes by the dominated convergence theorem. For the second term in (\ref{Dbar:split}), we apply integration by parts, to get 
\[
\lim_{y\to -\infty}{\overline{\DD^\alpha}[\varphi](y)} =
\lim_{y\to -\infty}\left(\left.\frac{\varphi(z+y)}{z^\alpha}\right|_r^\infty\right) +\alpha \lim_{y\to -\infty} \int_{r}^{\infty}{\frac{\varphi(z+y)}{z^{\alpha+1}}}\ dz. 
\]

The first term in the last identity clearly vanishes, and the second does too, again, by applying the dominated convergence theorem. This implies that (\ref{Dbar:split}) vanishes, and so does the last term in (\ref{D:Manipu2:bar}). 

It remains to prove the equivalent integral representations of $\overline{\DD^\alpha}$ and $\partial_x\overline{\DD^\alpha}$. These are shown as in the proof of Lemma~\ref{equiv:repre}, we do not write it here.
\end{proof}

We can now prove the entropy inequality for continuous entropies:
\begin{theorem}[Weak viscous entropy inequality]\label{entropy:smooth} Given $\e>0$, $\eta\in C(\R)$ convex and $u^\e\in C^2_b((0,\infty)\times\R)$ a solution of (\ref{Problem2}), then for all 
$\varphi \in C_c^\infty((0,\infty)\times \R)$
\begin{equation}\label{entropy:eps}
  \int_{0}^{\infty}\int_{\R} 
\Big( 
\eta(u^\e(t,x)) \partial_t\varphi(t,x) + q(u^\e(t,x)) \partial_x \varphi(t,x) 
+ \e \, \eta(u^\e(t,x)) \partial_x \overline{\DD^\alpha}[\varphi(t,\cdot)](x)
\Big) \, dx\, dt\geq 0,
  \end{equation}
where $q$ is given by
\begin{equation}\label{entropy:flux}
q(u)=f'(u)\eta(u)-f'(0)\eta(0)-\int_0^u f''(z)\eta(z)\,dz.
\end{equation}
\end{theorem}

\begin{proof}
First we assume that $\eta\in C^2(\R)$. We then notice that 
\begin{equation}\label{convex:multi}
\partial_x \DD^\alpha[\eta(\varphi)](x) \geq \eta'(\varphi) \partial_x \DD^\alpha[\varphi](x).
\end{equation}
This follows from the convexity of $\eta$ applied in the representation of $\partial_x \DD^\alpha[\eta(\varphi)](x)$ given by 
Lemma~\ref{equiv:repre} (\ref{integralformula1}). Now, using this and multiplying the equation in (\ref{Problem2}) by $\eta'(u^\e(t,x))$ gives the entropy-type inequality
\begin{equation}\label{entropyinequality2}
  \partial_t \eta(u^\e(t,x)) + \partial_x q(u^\e(t,x))
  \leq \e \, \partial_x \DD^\alpha[\eta(u^\e(t,\cdot))](x).
\end{equation}

We need a weak version of (\ref{entropyinequality2}), thus we multiply it by a non-negative test function $\varphi \in C_c^\infty((0,\infty)\times \R)$ and integrate over the whole domain. After integration by parts, we get
\[
\int_{0}^{\infty}\!\!\!\!\!\int_{\R} 
\Big( 
\eta(u^\e(t,x)) \partial_t\varphi(t,x) + q(u^\e(t,x)) \partial_x \varphi(t,x) 
+ \e \ \partial_x \DD^\alpha[\eta(u^\e(t,\cdot))](x) \varphi(t,x)
\Big) \, dx\, dt\geq 0.
\]
With application of Lemma~\ref{integrationbyparts} we conclude (\ref{entropy:eps}).

It remains to show the result for continuous convex entropies. 
Let $\eta\in C(\R)$ convex, and let $\omega_n\in C_c^\infty(\R)$ such that 
$\omega_n(x)=n\omega(nx)$ with $\omega\geq 0$, $\int_\R \omega=1$, then the functions 
$\eta_n=\omega_n\ast \eta\in C^2(R)$ are convex and converge locally uniformly 
to $\eta$. 

Associated to each element of this sequence of entropies we have an entropy flux $q_n(x)=\int_0^x f'(z)\eta_n'(z)\,dz$. 
Integrating by parts and taking the limit $n\to\infty$ one obtains that $q_n$ converges locally uniformly to (\ref{entropy:flux}). 
Since the inequality (\ref{entropy:eps}) is satisfied for all smooth entropy pairs $(\eta_n,q_n)$, then passage to the limit as
 $n\to\infty$ gives the desired inequality.
\end{proof}

We then show the $L^1$-contraction property:
\begin{theorem}[$L^1$-contraction]\label{L1:contraction}
For all $\e >0$, given $u_0^\e$, $v_0^\e \in L^\infty(\R)$ such that $u_0^\e-v_0^\e \in L^1(\R)$, 
let $u^\e$ and $v^\e$ be the corresponding mild solutions of 
(\ref{Problem1}) with these initial conditions, respectively. Then, for all $t\in (0,\infty)$, $u^\e(t,\cdot)-v^\e(t,\cdot) \in L^1(\R)$, and  
\begin{equation*}
\|u^\e(t,\cdot)-v^\e(t,\cdot)\|_{1} \leq \|u_0^\e - v_0^\e\|_{1}.
\end{equation*}
\end{theorem}

We recall that uniqueness immediately follows from this theorem.

\begin{proof}
For simplicity of notation and without loss of generality, we take $\e=1$ throughout this proof. 
We thus skip the $\e$ dependency in the notation of the solutions. The proof is based on Kruzhkov's doubling variable technique 
and on specific choices of test functions of the right weak entropy inequality. Similar arguments can be found in \cite{KarlsenRisebro} and \cite{Cifani}. 

First, we show that $u(t,\cdot)-v(t,\cdot) \in L^1(\R)$. This follows from the mild formulation (\ref{solutioneq}), using that $u_0-v_0\in L^1(\R)$, that 
$u(t,\cdot)$, $v(t,\cdot) \in  L^\infty(\R)$ and Proposition~\ref{K:prop} {\it (iv)}, so that:
\begin{align}
\|u(t,\cdot)-v(t,\cdot)\|_1 & \leq  \| K(t,\cdot)\|_1 \|u_0-v_0\|_1 
\nonumber	\\
& \qquad + C\left( \|u(t,\cdot)\|_\infty,\|v(t,\cdot)\|_\infty\right)
\int_0^t\int_\R |\partial_x K(t-s,y)|dy\,ds \nonumber \\
 & \leq \|u_0-v_0\|_1  \label{L1:uv}
\\
& \qquad + C\left( \|u(t,\cdot)\|_\infty,\|v(t,\cdot)\|_\infty\right)  \int_0^t \frac{B_1}{(t-s)^{\frac{2}{1+\alpha}}}  \int_\R \frac{dy}{1 + (t-s)^{-\frac{3}{1+\alpha}}|y|^3 } ds  \nonumber \\ 
 & = \|u_0-v_0\|_1 
  \nonumber \\
& \qquad + C\left( \|u(t,\cdot)\|_\infty,\|v(t,\cdot)\|_\infty\right) 
B_1\frac{1+\alpha}{\alpha} t^{\frac{\alpha}{1+\alpha}} \int_\R \frac{1}{1+|z|^3}dz <\infty. \nonumber
\end{align}

We now proceed as in the proof of Lemma~\ref{entropy:smooth}, but we leave the terms with integrand of the form 
$\eta' \partial_x \DD^\alpha$ as such, then we can argue, similarly for just continuous entropies, so that instead 
of (\ref{entropy:eps}) we obtain for any finite $T>0$
\begin{equation}\label{entropy:1}
    \int_{0}^{T}\int_{\R} 
\Big( 
\eta(u(t,x)) \partial_t\varphi(t,x) + q(u(t,x)) \partial_x \varphi(t,x) 
+ \eta'(u(t,x)) \partial_x \DD^\alpha[u(t,\cdot)](x)\varphi(t,x)
\Big) dx\, dt \geq 0.
 \end{equation}

Let $\psi = \psi(t,x,s,y) \in C_c^\infty((0,T) \times \R\times(0,T) \times\R)$ be a non-negative test function. We consider the family of Kruzhkov's entropies 
$\eta_v(u(t,x))= |u(t,x)-v(s,y)|$ and $\eta_u(v(s,y))= |v(s,y)-u(t,x)|$, respectively, and write the corresponding entropy inequality (\ref{entropy:1}) for $u(t,x)$ and $v(s,y)$ separately.
Then, integrating over $(s,y)\in (0,T)\times\R$ and over $(t,x)\in (0,T)\times\R$, respectively, each of these entropy inequalities, we add them up and apply Fubini's theorem, to obtain
\begin{align}\label{3.3}
\int_{0}^{T}\!&\int_{\R}\!\int_{0}^{T}\!\int_{\R}\Big\{|u(t,x)-v(s,y)| (\partial_t + \partial_s) \psi(t,x,s,y) \cr
& + \sign\left(u(t,x)-v(s,y)\right)\left(f(u(t,x)) - f(v(s,y)) \right) (\partial_x + \partial_y) \psi(t,x,s,y)\\
  &+\sign\left(u(t,x)-v(s,y)\right)\left(\partial_x \DD^\alpha[u(t,\cdot)](x)-\partial_y \DD^\alpha[v(s,\cdot)](y)\right)\psi(t,x,s,y)\Big\}
  dx\,dt\,dy\,ds \geq 0.\nonumber
\end{align}

In order to find a suitable entropy inequality, we have to manipulate the last term of (\ref{3.3}), 
\begin{equation}\label{lalala}
I:=\int_{0}^{T}\!\!\!\!\int_{\R}\!\int_{0}^{T}\!\!\!\!\int_{\R}
\sign(u(t,x)-v(s,y))\,\left(\partial_x\DD^\alpha[u(t,\cdot)](x)-\partial_y\DD^\alpha[v(s,\cdot)](y)\right)\psi(t,x,s,y)
dx\,dt\,dy\,ds.
\end{equation}
We use Lemma~\ref{equiv:repre} in the integrand of $I$:
\begin{align}\label{L1:key:est}
&\!\!\mbox{sgn}(u(t,x)-v(s,y))\,\left(\partial_x\DD^\alpha[u(t,\cdot)](x)-\partial_y\DD^\alpha[v(s,\cdot)](y)\right)
	\cr
&=d_{\alpha+2}\,\mbox{sgn}(u(t,x)-v(s,y))
	\cr
&\cdot\int_{-\infty}^0 \frac{u(t,x+z)-v(s,y+z)-(u(t,x)-v(s,y))-(u_x(t,x)+v_y(s,y))z}{|z|^{\alpha+2}}dz
	\\
&\leq d_{\alpha+2}\int_{-\infty}^0 \frac{|u(t,x+z)-v(s,y+z)|-|u(t,x)-v(s,y)|-
	(\partial_x+\partial_y)\left(|u(t,x)-v(s,y)|\right)z}{|z|^{\alpha+2}}dz. \nonumber
\end{align}


\noindent For simplicity of notation, we define the following operator acting on functions of two o more variables:
\begin{equation}\label{def:new:D}
\DD_{x,y}^{\alpha+1}[g](x,y):=d_{\alpha+2}\int_{-\infty}^0\frac{g(x+z,y+z)-g(x,y)-(\partial_x+\partial_y)g(x,y) \,z}{|z|^{\alpha+2}}dz.
\end{equation}
We can rewrite the estimate on $I$ based on (\ref{L1:key:est}) as
\begin{equation}\label{est:i}
I\leq \int_{0}^{T}\!\int_{\R}\!\int_{0}^{T}\!\int_{\R}
\DD_{x,y}^{\alpha+1}[|u(t,\cdot)-v(s,\cdot)| ](x,y)  \psi(t,x,s,y) dx\,dt\,dy\,ds.
\end{equation}

It is now convenient to split the operator (\ref{def:new:D}) into two integrals. For any $r>0$, we write
\[
\DD_{x,y}^{\alpha+1}[|u(t,\cdot)-v(s,\cdot)|](x,y)=( {}_r\DD_{x,y}^{\alpha+1} + {}^r\DD_{x,y}^{\alpha+1})[|u(t,\cdot)-v(s,\cdot)|](x,y)
  \]
  with, for a function $g(x,y)$, 
  \[
  {}_r\DD_{x,y}^{\alpha+1} [g](x,y)=d_{\alpha+2} \int_{-r}^0
  \frac{g(x+z,y+z)-g(x,y)-(\partial_x+\partial_y)g(x,y) \,z}{|z|^{\alpha+2}}dz.
  \]
  and with the obvious definition for ${}^r\DD_{x,y}^{\alpha+1}$.

  With this splitting, from (\ref{3.3}) and (\ref{est:i}), we obtain the following entropy type inequality:
\begin{equation}\label{L1:entropy:ineq}
\begin{split}
&\int_{0}^{T}\!\int_{\R}\!\int_{0}^{T}\!\int_{\R}\Big\{|u(t,x)-v(s,y)| (\partial_t + \partial_s) \psi(t,x,s,y) \\
&+  \sign(u(t,x)-v(s,y))\left(f(u(t,x)) - f(v(s,y)\right) (\partial_x + \partial_y) \psi(t,x,s,y)\\
  &+\left( {}_r\DD_{x,y}^{\alpha+1}[|u(t,\cdot)-v(s,\cdot)|](x,y)+
    {}^r\DD_{x,y}^{\alpha+1}[|u(t,\cdot)-v(s,\cdot)|](x,y)  \right)\psi(t,x,s,y)\Big\}
  dx\,dt\,dy\,ds \geq 0.
\end{split}
\end{equation}
We observe that, since $r>0$, the last term can be seen as three finite integrals. Using Fubini's theorem and the change of variables $(x+z,y+z,z)\to (x,y,-z)$ in the first, Fubini's theorem and the change of variable $z\to-z$ in the second and the third, and also integration by parts in the third, we obtain 
\[
\begin{split}
&\int_{0}^{T}\!\int_{\R}\!\int_{0}^{T}\!\int_{\R}
 {}^r\DD_{x,y}^{\alpha+1}\left[|u(t,\cdot)-v(s,\cdot)|\right](x,y) \psi(t,x,s,y)dx\,dt\,dy\,ds\\
&=\int_{0}^{T}\!\int_{\R}\!\int_{0}^{T}\!\int_{\R}
|u(t,x)-v(s,y)| {}^r\overline{\DD}_{x,y}^{\alpha+1}[\psi(t,\cdot,s,\cdot)](x,y)dx\,dt\,dy\,ds,
\end{split}
\]
with
\begin{equation}\label{def:new:Dadjoint}
{}^r\overline{\DD}_{x,y}^{\alpha+1}[g](x,y)=
d_{\alpha+2} \int^{\infty}_{r}  \frac{g(x+z,y+z)-g(x,y)-(\partial_x+\partial_y)g(x,y)\,z}{|z|^{\alpha+2}}  dz. 
\end{equation}
Now, taking the limit $r\to 0$ in (\ref{L1:entropy:ineq}), with the last term as above, we finally get, 
by the dominated convergence theorem, the entropy inequality 
\begin{equation}\label{L1:entropy:ineq:def}
\begin{split}
&\int_{0}^{T}\!\int_{\R}\!\int_{0}^{T}\!\int_{\R}
\Big(|u(t,x)-v(s,y)| (\partial_t + \partial_s) \psi(t,x,s,y) \\
& + \sign\left(u(t,x)-v(s,y)\right)\left(f(u(t,x)) - f(v(s,y)) \right) (\partial_x + \partial_y) \psi(t,x,s,y)\\
  &+|u(t,x)-v(s,y)|\, \overline{\DD}_{x,y}^{\alpha+1}[\psi(t,\cdot,s,\cdot)](x,y) \Big)
  dx\,dt\,dy\,ds \geq 0
\end{split}
\end{equation}
where 
\[
\overline{\DD}_{x,y}^{\alpha+1}[g](x,y)={}^0\overline{\DD}_{x,y}^{\alpha+1}[g](x,y).
\]

We now specify the test functions $\psi$ in order to derive the $L^1$-contraction from (\ref{L1:entropy:ineq:def}). We take:
\[
\psi(t,x,s,y) =\omega_\rho\left(\frac{s-t}{2}\right) \omega_\rho\left(\frac{y-x}{2}\right) \varphi\left(\frac{t+s}{2},\frac{x+y}{2}\right),
\]
where for any $\rho>0$, 
and $\omega_\rho(s) = \omega(s/\rho)/\rho$ for a non-negative $\omega \in C_c^\infty(\R)$ satisfying, $\omega(-s) = \omega(s)$, $\omega(0)=1$, $\omega(s)=0$ for all $|s|\geq 1$ and $\int_{\R}{\omega(s)\ ds}=1$. And, for the moment we ask $\varphi \in C_c^\infty((0,\infty)\times\R)$ to be non-negative, we will specify the choice of this function later.

In this way, we obtain that
\begin{eqnarray*}
&(\partial_t + \partial_s)\psi(t,x,s,y) =\omega_\rho\left(\frac{s-t}{2}\right) \omega_\rho\left(\frac{y-x}{2}\right) (\partial_t + \partial_s) \varphi\left(\frac{t+s}{2},\frac{x+y}{2}\right), 
\\
&(\partial_x + \partial_y)\psi(t,x,s,y) = \omega_\rho\left(\frac{s-t}{2}\right)\omega_\rho\left(\frac{y-x}{2}\right) (\partial_x + \partial_y) \varphi\left(\frac{t+s}{2},\frac{x+y}{2}\right), \\
&\overline{\DD}_{x,y}^{\alpha+1}[\psi(t,\cdot,s,\cdot)](x,y)=\omega_\rho\left(\frac{s-t}{2}\right) \omega_\rho\left(\frac{y-x}{2}\right) \overline{\DD}_{x,y}^{\alpha+1}\left[\varphi\left(\frac{t+s}{2},\frac{\cdot+\cdot}{2}\right)\right](x,y). 
\end{eqnarray*}

With the changes of variables:
\[
r=\frac{s-t}{2},\ r'=\frac{s+t}{2}, \ z=\frac{y-x}{2},\ z'=\frac{x+y}{2}
\]
we obtain
\begin{eqnarray*}
(\partial_t + \partial_s)\varphi\left(\frac{t+s}{2},\frac{x+y}{2}\right)=\partial_{r'}\varphi(r',z') \\
 (\partial_x + \partial_y) \varphi\left(\frac{t+s}{2},\frac{x+y}{2}\right)=\partial_{z'}\varphi(r',z')  \\
\overline{\DD}_{x,y}^{\alpha+1}\left[\varphi\left(\frac{t+s}{2},\frac{\cdot+\cdot}{2}\right)\right](x,y)= \partial_{z'} \overline{\DD^\alpha}[\varphi(r',\cdot)](z'),
\end{eqnarray*}
(see the last statement of Lemma~\ref{integrationbyparts} for the expression of  $\partial_{z'} \overline{\DD^\alpha}$).
With these test functions and the above change of variables, (\ref{L1:entropy:ineq:def}) becomes:
\begin{equation}\label{3.7}
\begin{split}
&
\int_{0}^{T}\!\int_{\R}\!\int_{-\frac{T}{2}}^{\frac{T}{2}}\!\int_\R
\omega_\rho(r) \omega_\rho(z)\Big(|u(r'-r,z'-z)-v(r+r',z'+z)|\partial_r'\varphi(r',z') \\
  +& \sign(u(r'-r,z'-z)-v(r+r',z'+z))(f(u(r'-r,z'-z)) - f(v((r+r',z'+z))) \partial_{z'}  \varphi(r',z')\\
  +&  |u(r'-r,z'-z)-v(r+r',z'+z)|
\partial_{z'} \overline{\DD^\alpha}[\varphi(r',\cdot)](z') \Big) dz\,dr\,dz'\,dr' \geq 0.
\end{split}
\end{equation}
 
Applying the Lebesgue differentiability theorem, taking the limit  $\rho \to 0^+$, (\ref{3.7}) reduces to 
\begin{eqnarray}\label{3.8}
\nonumber \int_{0}^{T}\int_{\R} & |u(t,x)-v(t,x)| \partial_t \varphi(t,x) + \sign(u(t,x)-v(t,x))(f(u(t,x))-f(v(t,x)))\partial_x \varphi(t,x)
\\
& +  |u(t,x)-v(t,x)| \partial_x \overline{\DD^\alpha}[\varphi(t,\cdot)](x) \ dx\,dt \geq 0,
\end{eqnarray}
where we have renamed the variables ($(t,x)$ instead of $(r',z')$).

In order to conclude the proof, we now choose for $\mu$, $R>0$,
$\varphi(t,x)=\phi_\mu(x)\Theta_R(t)$ where, 
\[
\phi_\mu(x)= \int_{\R} \omega(x-y)\chi_{|y|<\mu} dy =\int_{x-\mu}^{x+\mu} \omega(z)dz,
\]
thus all derivatives of $\phi_\mu$ are bounded uniformly in $\mu$ and vanish for all $||x|-\mu|>1$. 
And, for any pair $0<R<t_1<t_2$, we choose
\[
\Theta_R (t)=\int_{-\infty}^t(\omega_R(\tau-t_1) -\omega_R(\tau-t_2))d\tau.
\]

First, we observe that taking the limit $\mu\to\infty$, 
the inequality (\ref{3.8}) reduces to 
\begin{equation}\label{3.10}
\int_{0}^{T}\int_{\R}|u(t,x)-v(t,x)| \Theta_R'(t) \ dxdt \geq 0.
\end{equation}
Indeed, concerning the flux-term in (\ref{3.8}), we find that
\begin{eqnarray*}
&\displaystyle\int_{0}^{T}\!\int_{\R} \sign(u(t,x)-v(t,x)) 
\left(f(u(t,x))-f(v(t,x))\right) \partial_x \varphi(x,t) dx\,dt \\
&\displaystyle\leq L \|\Theta_R\|_{L^\infty(0,\infty)} 
\int_{0}^{T}\!\int_{\R} |u(t,x)-v(t,x)|\, |\omega(x+\mu)-\omega(x-\mu)|dx\,dt
 \xrightarrow{\mu\to \infty} 0.
\end{eqnarray*}
Here, we have applied the dominated convergence theorem, since $u-v\in L^1$ and 
$|\omega(x+\mu)-\omega(x-\mu)|\to 0$ as $\mu\to \infty$ for all $x\in \R$.

The term in (\ref{3.8}) containing the non-local operator also tends to zero as $\mu\to \infty$. To see this,
 note that $\left|\partial_x\overline{\DD^\alpha}[\phi_\mu](x)\right|$ is uniformly bounded in $\mu$, 
since, arguing as for the operator $\partial_x\DD^\alpha$, one obtains for some $C>0$ independent of $\mu$
\[
\left|\partial_x\overline{\DD^\alpha}[\phi_\mu](x)\right| \leq 
C_0 \max\{\|\phi_\mu\|_\infty, \|\phi_\mu'\|_\infty,\|\phi_\mu''\|_\infty \}\leq C.
\]
Now, by integrability of $u-v$ in $x$ and H\"older's inequality, we obtain
\[
\displaystyle\int_{0}^{T} \int_{\R}|u(t,x)-v(t,x)| \left|\partial_x \overline{\DD^\alpha}[\phi_\mu](x)\right| dx\,dt 
\leq T C \sup_{t\in(0,T)}\|u(t,\cdot)-v(t,\dot)\|_1 .
\]

Observe that 
\[
\partial_x \overline{\DD^\alpha}[\phi_\mu](x)
 = d_\alpha \int_x^\infty \frac{\omega(z+\mu)-\omega(z-\mu)}{(z-x)^\alpha}dz\to 0 \quad  \mbox{as} \quad \mu\to\infty \quad\mbox{a.e.}
\]
because we can take, for each $x$, $\mu$ large enough so that $x+\mu>1$ and $x-\mu<-1$:
\[
\partial_x \overline{\DD^\alpha}[\phi_\mu](x)
 = -d_\alpha \int_x^\infty \frac{\omega(z-\mu)}{(z-x)^\alpha}dz = -d_\alpha \int_{-1-x+\mu}^{1-x+\mu} \frac{\omega(z+x-\mu)}{z^\alpha}dz,
\]
and we can apply the dominated convergence theorem.
With this, we can conclude, also by the dominated convergence theorem, that
\[
\lim_{\mu\to \infty} \int_{0}^{\infty}\!\int_{\R}|u(t,x)-v(t,x)| \left|\partial_x \overline{\DD^\alpha}[\phi_\mu](x)\right| dx\,dt=0.
\]

We use now the definition of $\Theta_R$ in (\ref{3.10}). Since 
$\Theta'_R(t)= \omega_R(t-t_1)- \omega_R(t-t_2)$, we have
\begin{eqnarray*}
\int_{0}^{T}\!\int_{\R}|u(t,x)-v(t,x)| \omega_R(t-t_2) \ dx\,dt \leq \int_{0}^{T}\!\int_{\R}|u(t,x)-v(t,x)|\omega_R(t-t_1)  dx\,dt,
\end{eqnarray*}
that can be written as
\begin{equation}\label{L1:con:t1t2}
\begin{split}
\frac{1}{R}\int_{-R}^{R}\int_{\R}|u(s+t_2,x)-v(s+t_2,x)|\, \omega\left(\frac{s}{R}\right) dx\,ds\\ \leq 
\frac{1}{R} \int_{-R}^{R}\int_{\R}|u(s+t_1,x)-v(s+t_1,x)|\, \omega\left(\frac{s}{R}\right) dx\,ds.
\end{split}
\end{equation}
We now take the limit $R\to 0$ in (\ref{L1:con:t1t2}), and by the Lebesgue differentiability theorem we obtain 
\[
\|(u-v)(t_2,\cdot)\|_{1} \leq \|(u-v)(t_1,\cdot)\|_{1}.
\]

Finally, the theorem follows by renaming $t_2$ to $t$ and taking the limit $t_1\to 0$, since using (\ref{L1:uv}), we obtain
\[
\limsup_{t_1\to 0^+}\|(u-v)(t_1,\cdot)\|_{1}\leq \|u_0-v_0\|_{1},
\]
thus the result follows.
\end{proof}

\section{The vanishing viscosity limit}\label{sec:limit}
In this section we show that in the limit when $\e\to 0^+$ in 
(\ref{Problem2}) we obtain the entropy solution associated to (\ref{Problem3}). We follow a doubling variable 
technique as in \cite{Droniou2}, but with the pertinent changes due to the different nonlocal operator in the viscous term.

We need the following technical Lemma: 

\begin{lemma}\label{L1integrable}
Let $\varphi \in C_c^\infty((0,\infty)\times \R)$, then the maps
\[
t \in (0,\infty) \mapsto \DD^\alpha[\varphi(t,\cdot)]\in L^1(\R) 
\, \quad t \in (0,\infty) \mapsto \overline{\DD^\alpha}[\varphi(t,\cdot)]\in L^1(\R) \] 
and
\[
t \in (0,\infty) \mapsto \partial_x \DD^\alpha[\varphi(t,\cdot)]\in L^1(\R) \, \quad t \in (0,\infty) \mapsto \partial_x \overline{\DD^\alpha}[\varphi(t,\cdot)]\in L^1(\R)
\] 
are continuous, and as functions of $(t,x)$, $\DD^\alpha[\varphi(t,\cdot)](x)$, $\overline{\DD^\alpha}[\varphi(t,\cdot)](x)$, $\partial_x\DD^\alpha[\varphi(t,\cdot)](x)$ and $\partial_x \overline{\DD^\alpha}[\varphi(t,\cdot)](x)$ are integrable over 
$(0,\infty)\times \R$. Moreover, there exists $C_\alpha>0$ such that  
\begin{equation}\label{norminequality}
\|\DD^\alpha \varphi\|_{1} \leq C_\alpha \left(\|\varphi'\|_{1} + \|\varphi\|_{1} \right), \quad \|\overline{\DD^\alpha} \varphi\|_{1} \leq C_\alpha \left(\|\varphi'\|_{1} + \|\varphi\|_{1} \right),
\end{equation}
and
\begin{equation}\label{norminequality2}
\|\partial_x \DD^\alpha \varphi\|_{1} \leq C_\alpha \left(\|\varphi''\|_{1} + \|\varphi'\|_{1} \right),\quad \|\partial_x \overline{ \DD^\alpha} \varphi\|_{1} \leq C_\alpha \left(\|\varphi''\|_{1} + \|\varphi'\|_{1} \right).
\end{equation}

\end{lemma}
\begin{proof} We only prove the statements for $\overline{\DD^\alpha}$ and $\partial_x \overline{\DD^\alpha}$, the rest of the proofs are analogous.
  
By the properties of $\varphi$, we can write, for an arbitrary $r>0$,  
\[
 \overline{\DD^\alpha}[\varphi](x)=
d_\alpha \left(
\int^0_{-r}  \frac{\varphi'(x-z)}{|z|^\alpha}\,dz +\int^{-r}_{-\infty} \frac{\varphi'(x-z)}{|z|^\alpha}\,dz
\right),
\]
and
\[
\partial_x \overline{\DD^\alpha}[\varphi](x)=
d_\alpha \left(
\int^0_{-r} \frac{\varphi''(x-z)}{|z|^\alpha}\,dz +\int^{-r}_{-\infty} \frac{\varphi''(x-z)}{|z|^\alpha}\,dz
\right).
\]
However, we notice that, by integration by parts (see also \cite{dlHC}) and that $\varphi \in C_c^\infty(\R)$ we can write:
\begin{equation}\label{alt:opr}
\begin{split}
\overline{\DD^\alpha} [\varphi](x) &= d_\alpha \int^0_{-r}  \frac{\varphi'(x-z)}{(-z)^\alpha}\,dz
- d_{\alpha+1} \int^{-r}_{-\infty}\frac{\varphi(x-z) }{(-z)^{\alpha+1}}\,dz + d_\alpha \frac{\varphi(x+r)}{ r^{\alpha}}
\end{split}
\end{equation}
and
\begin{equation}\label{alt:opr2}
\begin{split}
\partial_x \overline{\DD^\alpha} [\varphi](x) &= d_\alpha \int^0_{-r}  \frac{\varphi''(x-z)}{(-z)^\alpha}\,dz
-d_{\alpha+1} \int^{-r}_{-\infty}\frac{\varphi'(x-z) }{(-z)^{\alpha+1}}\,dz + d_\alpha \frac{\varphi'(x+r)}{ r^{\alpha}}.
\end{split}
\end{equation}
Now, taking $r=1$ for definiteness, applying Young's inequality in the first and second terms of the right-hand side of (\ref{alt:opr}) and of (\ref{alt:opr2}),
 we obtain  (\ref{norminequality}) and (\ref{norminequality2}).
If $\varphi \in C_c^\infty((0,\infty)\times \R)$, then $t \in (0,\infty) \mapsto \varphi'(t,\cdot) \in L^1(\R)$ and 
$t \in (0,\infty) \mapsto \varphi''(t,\cdot) \in L^1(\R)$ are continuous. We use the inequality (\ref{norminequality}) 
and the linearity of $\overline{  \DD^\alpha}$ to get that the function 
$t \in (0,\infty) \mapsto \overline{\DD^\alpha}[\varphi(t,\cdot)]$ is continuous. 
In particular, since $\varphi(t,\cdot)=0$ for $t$ large enough, we get that 
$(t,x) \mapsto \overline{\DD^\alpha}[\varphi(t,\cdot)](x)$ is integrable on $(0,\infty)\times \R$. A similar argument 
is applied to $\partial_x \overline{\DD^\alpha}[\varphi]$ to conclude the proof.
\end{proof}

We can now prove the main theorem of this section.
\begin{theorem}\label{vanishingregularization}
\begin{itemize}
\item[(a)]Let $u_0 \in L^\infty(\R)$. The mild solution to (\ref{Problem2}), $u^\e$, converges, as $\e \to 0$, 
to the entropy solution of (\ref{Problem3}) $u$ in $C([0,T]; L^1_{loc}(\R))$ for all $T>0$.

\item[(b)] Let $u_0 \in L^\infty(\R)\cap BV(\R)$, $u^\e$ be the mild solution to (\ref{Problem2}) and $u$ be the entropy solution of (\ref{Problem3}). Then, for all $t\in [0,T]$, and $\e>0$ small enough there exists a constant $C>0$ such that
\begin{equation}\label{epsilontimevariation}
\|u^\e(t,\cdot)- u(t,\cdot )\|_{1} \leq C \left(\e \,t \right)^{\frac{1}{\alpha+1}}|u_0|_{BV}.
\end{equation}
In particular, for all $T>0$, $\|u^\e - u\|_{C([0,T]; L^1(\R))}= O\left(\e^{\frac{1}{\alpha+1}}\right)$ as $\e \to 0^+$.
\end{itemize}
\end{theorem}

\begin{proof}
First, we recall that the entropy $u$ solution of (\ref{Problem3}) is in $C([0,T]; L^1_{loc}(\R))$ and satisfies (\ref{entropy:eps})
 with $\e=0$ (see \cite{Serre}). 


For all $\e>0$ let $u^\e \in C_b^\infty((0,\infty)\times \R)$ be the regular mild solution of (\ref{Problem2}) 
with the same initial condition for all $\e$. Then, each $u^\e$ satisfies Theorem~\ref{entropy:smooth}.
These inequalities can be written for test functions of four variables (thus doubling the variables), 
$\psi(t,x,s,y)\in C_c^\infty((0,\infty)\times \R\times (0,\infty)\times \R)$. 
Indeed, we have for two entropy pairs $(\eta,q)$ and $(\eta_0,q_0)$, 
\begin{align*}
\int_{0}^{\infty}\!\int_{\R} \! \int_{0}^{\infty}\!\int_{\R} 
\Big( &
\eta(u^\e(t,x)) \partial_t\psi(t,x,s,y) + q(u^\e(t,x)) \partial_x \psi(t,x,s,y) 
	\\
&+ \e \, \eta(u^\e(t,x)) \partial_x \overline{\DD^\alpha}[\psi(t,\cdot,s,y)](x)\\
&+
 \eta_0(u(s,y)) \partial_s\psi(t,x,s,y) + q_0(u(s,t)) \partial_y \psi(t,x,s,y) 
\Big) dx\, dt\, dy\, ds  \geq 0,
\end{align*}

where we have applied Fubini's theorem in the last integral. We take the Kruzhkov entropies 
 $\eta(u^\e(t,x))= |u^\e(t,x)-u(s,y)|$ and $\eta_0(u(s,y))= |u^\e(t,x)-u(s,y)|$, and  
a test function of the form
\[
 \psi(t,x,s,y) =  \theta_\mu(s-t)\omega_\rho (y-x)\varphi(t,x)
 \]
where, for $\rho>0$, we take $\omega_\rho \in C_c^\infty(\R)$, as in the proof of Theorem~\ref{L1:contraction}. In particular, $\int_\R\omega_\rho(y) dy=1$ and 
$\mbox{supp}(\omega_\rho) \subset (-\rho,\rho)$. For $\mu>0$, we take $\theta_\mu \in C_c^\infty(\R)$ such that $\int_{0}^{\infty}\theta_\mu(s) ds=1$ 
with $\mbox{supp}(\theta_\mu) \subset (0,\mu)$, and that $\mu\theta_\mu(\mu/2)=1$ (for example taking $\theta_\mu(x)=\omega_{\mu/2}(x+\mu/2)$). Observe that then, for any $x\in\R$ and $t>0$, $\int_{0}^{\infty}\theta_\mu(s-t)ds= \int_{\R}\omega_\rho(y-x) dy=1$.
We take $\varphi \in C_c^\infty((0,\infty)\times\R)$ to be a non-negative function that will be specified later.

With these choices we get
\begin{equation}\label{doublingequation1}
\begin{split}
\int_{0}^{\infty}\!\int_{\R}\!\int_{0}^{\infty}\!\int_{\R} & \Big( 
|u^\e(t,x)-u(s,y)| \theta_\mu(s-t) \omega_\rho(y-x) \partial_t \varphi(t,x) \\
&+ \sign(u^\e(t,x)-u(s,y))\left(f(u^\e(t,x))- f(u(s,y))\right)
 \theta_\mu(s-t)\omega_\rho(y-x)\partial_x \varphi(t,x) \\
& +\e   |u^\e(t,x)-u(s,y)|\theta_\mu(s-t)
  \partial_x \overline{\DD^\alpha}[\omega_\rho(y-\cdot) \varphi(t,\cdot)](x)  \Big)
 dx\,dt\,dy\,ds \geq 0.
\end{split}
\end{equation}

 We then estimate the following terms separately:
\begin{equation}\label{doublingequation2}
 \begin{split}
 I_1&:=\int_{0}^\infty\!\int_{\R}\!\int_{0}^\infty\!\int_{\R} |u^\e(t,x)-u(s,y)| \theta_\mu(s-t) \omega_\rho(y-x)  \partial_t \varphi(t,x) dx\,dt\,dy\,ds,  \\
 I_2&:= \int_{0}^\infty\!\int_{\R}\!\int_{0}^\infty\!\int_{\R} \sign(u^\e(t,x)-u(s,y))\left(f(u^\e(t,x))- f(u(s,y))\right)
 	\\
 & \qquad\qquad\qquad\qquad\qquad\qquad\cdot\theta_\mu(s-t) \omega_\rho(y-x)  \partial_x \varphi(t,x)dx\,dt\,dy\,ds, \\
 I_3&:= \e \int_{0}^\infty\!\int_{\R}\!\int_{0}^\infty\!\int_{\R} \ |u^\e(t,x)-u(s,y)|\theta_\mu(s-t)
  \partial_x \overline{\DD^\alpha}[\omega_\rho(y-\cdot) \varphi(t,\cdot)](x)  dx\,dt\,dy\,ds.
 \end{split}
 \end{equation}
For that we proceed as in \cite{Droniou1}. Suppose that for every $T>0$, $\mbox{supp}(\varphi) \subset (0,T]\times B$, for some ball $B \subset \R$, then
 \begin{equation}\label{I1inequality}
 \begin{split}
 & \left|I_1 - \int_{0}^\infty\!\int_{\R} |u^\e(t,x) - u(t,x)| \partial_t \varphi(t,x) dx\, dt \right|\\ 
\leq &  \int_{0}^T\!\int_{\R}\!\int_{0}^\infty\!\int_{\R}  
\Big| |u^\e(t,x)-u(s,y)|-|u^\e(t,x)-u(t,x)| \Big|
 \omega_\rho(y-x) \theta_\mu(s-t) \,|\partial_t \varphi(t,x)|dx\,dt\,dy\,ds \\
 \leq& \|\partial_t \varphi\|_{L^1(0,T; L^\infty(\R))} 
\sup_{0<t<T}\left\{\int_{0}^\infty\!\int_{\R}\!\int_B |u(t,x)- u(s,y)| \omega_\rho(y-x) \theta_\mu(s-t) dx\,dy\,ds\right\} ,
 \end{split}
 \end{equation}
where we have used that $\int_{0}^\infty\theta_\mu(s-t) ds=\int_{\R}\omega_\rho(y-x) dy=1$.

For the second integral, we apply that $f$ is locally Lipschitz continuous, and that $\|u^\e\|_{\infty}$, $\|u\|_{\infty}\leq \|u_0\|_{\infty}$.  
Thus, there exists a constant $L(\|u_0\|_{\infty})>0$ such that
 \[
\left|\sign(u^\e(t,x)-u(s,y)) \left(f(u^\e(t,x)) - f(u(s,y))\right)\right| \leq L(\|u_0\|_{\infty})\, |u^\e(t,x)-u(s,y)|,
\]
and therefore, by the triangle inequality after adding and subtracting $u(x,t)$ appropriately, we conclude that
\begin{equation}\label{I2inequality}
 \begin{split} |I_2| 
&\leq L(\|u_0\|_\infty) \left( \int_{0}^{\infty}\!\int_{\R} |u^\e(t,x) -u(t,x)| \, |\partial_x\varphi(t,x)| dx\,dt \right.
\\
&\left.+  \|\partial_x \varphi\|_{L^1(0,T; L^\infty(\R))}\sup_{0<t<T}\left\{ \int_{0}^{\infty}\!\int_{\R}\!\int_B |u(t,x)- u(s,y)| \omega_\rho(y-x) \theta_\mu(s-t) dx\,dy\,ds\right\}\right).
\end{split}
\end{equation}

For the last integral $I_3$ we get that
 \begin{equation}\label{I3inequality}
|I_3| \leq 2\e \ \|u_0\|_\infty \int_{\R}\!\int_{0}^{\infty}\int_{\R}
\left| \partial_x \overline{\DD^\alpha}[\omega_\rho(y-\cdot) \varphi(t,\cdot)](x)\right|  dx\,dt\,dy 
\leq \e C,
 \end{equation}
where $C$ is a constant proportional to
$$
\|u_0\|_\infty \sup_{0\leq t\leq T} \max\{\|\varphi(t,\cdot)\|_1,\|\partial_x\varphi(t,\cdot)\|_1,\|\partial_x^2\varphi(t,\cdot)\|_1\} (1+\rho).
$$
This is because $\varphi$ has compact support in $(0,T]\times \R$, and then $y\in[-a-\rho,a+\rho]$ for some $a>0$. Also Lemma~\ref{L1integrable} applies.

For brevity, and in view of (\ref{I1inequality}) and (\ref{I2inequality}), let us introduce the notation:
\begin{equation}\label{def:w:B}
w^B(\rho,\mu):=\sup_{0<t<T}\left\{ \int_{0}^{\infty}\!\int_{\R}\!\int_B |u(t,x)- u(s,y)| \omega_\rho(y-x) \theta_\mu(s-t) dx\,dy\,ds\right\}.
\end{equation}
We observe that, after the change of variables $z=y-x$ and $r=s-t+\mu/2$, leaving $x$ and $t$ unchanged, we get
\begin{equation}\label{est1:w:B}
w^B(\rho,\mu) = \sup_{0<t<T}\left\{\frac{1}{\rho}\frac{2}{\mu} \int_{0}^{\infty}\!\int_\R \!\int_B |u(t,x)- u(r+t-\mu/2,z+x)| \omega\left(\frac{z}{\rho}\right) \omega\left(\frac {2r}{\mu}\right) dx\,dz\,dr\right\},
\end{equation}
a form which is better suited to take limits of the parameters $\rho$ and $\mu$ to $0$, as we shall need to do below.

With this notation and summarising, the inequalities (\ref{I1inequality}), (\ref{I2inequality}) and (\ref{I3inequality}) applied in (\ref{doublingequation1}), 
give that there exist $L$, $C'$, $C>0$ such that
 \begin{equation}\label{viscous:limit1}
\int_{0}^{\infty}\!\int_{\R} |u^\e(t,x)-u(t,x)|\left(\partial_t \varphi(t,x) + L |\partial_x \varphi(t,x)|\right) dx\,dt +
 C' w^B(\rho,\mu) + \e \ C \geq 0,
 \end{equation}
 where 
\begin{equation}\label{viscous:limit1:constants}
\begin{split}
L &\ \propto \ \|u_0\|_\infty,\\
C'& \ \propto \  \|u_0\|_\infty\sup_{0\leq t\leq T}\|\partial_x\varphi(t,\cdot)\|_1,\\
C &\ \propto \ \|u_0\|_\infty \sup_{0\leq t\leq T} \max\{\|\varphi(t,\cdot)\|_1,\|\partial_x\varphi(t,\cdot)\|_1,\|\partial_x^2\varphi(t,\cdot)\|_1\},
\end{split}
\end{equation}
and none of these three constants depend on $\rho$ and $\mu$.

We now choose a $\varphi$ that is close to a solution of the factor $\partial_t \varphi(t,x) + L |\partial_x \varphi(t,x)|$. 
For any $T>0$, let $M>0$ be such that $M>L T$ and let also $\zeta_M \in C_c^\infty([0,\infty))$
 be non-increasing, with values in $[0,1]$ where $\zeta_M \equiv 1$ on $[0,M]$ and $\mbox{supp}(\zeta_M) \subset [0,M+1]$. 
We let also $\Theta \in C_c^\infty(0,T)$ with values in $[0,1]$, the precise choice of functions will be specified later. Then we take 
\[
\varphi(t,x)= \zeta_M(|x|+ L t)\Theta(t).
\]
Observe, that this is a non-negative function, that belongs to $C_c^\infty((0,\infty)\times\R)$ 
(the function $\Theta$ has its support in $[0,T)$ and $(t,x) \mapsto \zeta_M(|x| + L t)$ is regular on 
$[0,T)\times \R$ since, in a neighbourhood of $[0,T]\times\{0\}$, $\zeta_M(|x|+ L t)=1$) 
and $\mbox{supp}(\varphi) \subset (0,T)\times(-M-1,M+1)$, so we can take $B=(-M-1,M+1)$. 
This test function satisfies
 \begin{eqnarray*}
 \partial_t \varphi(t,x) &=& L \zeta_M'(|x| + L t) \Theta(t) + \zeta_M(|x|+ L t) \Theta'(t), \\
 |\partial_x \varphi(t,x)| &=& \left| \zeta_M'(|x|+ L t) \sign(x) \Theta(t)\right| = -\zeta_M'(|x|+ L t) \Theta(t),
 \end{eqnarray*}
 the last identity is true because $\zeta_M$ is non-increasing.

Now, substituting this into (\ref{viscous:limit1}) gives
 \begin{equation}\label{viscous:limit2}
\int_{0}^{T}\!\int_{\R} |u^\e(t,x)-u(t,x)| \zeta_M(|x| + L t) \Theta'(t)  dx\,dt
 + C' w^{B}(\rho,\mu) + \e \ C \geq 0.
 \end{equation}
With this choice of $\varphi$, the constants $C'$ and $C$ are of the form
\[
C \ \propto \ \max\{1,\|\zeta_M''\|_\infty\}\ 
\mbox{and}
\ 
C' \ \propto \  \int_0^T |\Theta(t)| dt.
\]

We now specify $\Theta(t)$. For every $t_0\in[0,T)$ we take a one parameter family of functions $\Theta(t)=\Theta_{t_0,\beta}(t)$ in the proofs below, where $\beta<T-t_0$, namely
  \begin{equation}\label{def:THETA}
  \Theta_{t_0,\beta}(t)=\int_t^\infty\theta_\beta(s-t_0)\,ds.
  \end{equation}

{\it Proof of (a)}
We first observe that applying the Lebesgue differentiability theorem to (\ref{est1:w:B}), we obtain that 
\begin{equation}\label{localw}
w^B(\rho,\mu)\to 0\quad \mbox{as}\quad (\rho,\mu)\to (0,0).
\end{equation}
We can even take $\mu=\rho$ and take the limit $\rho\to0^+$ in (\ref{viscous:limit2}). Then 
 \[
\int_{0}^{T}\!\int_{-(M+1)-Lt}^{M+1-Lt} |u^\e(t,x)-u(t,x)| \,\zeta_M(|x| + L t) \Theta'(t)  dx\,dt + \e \ C \geq 0.
 \]
Now we take for all $t_0\in[0,T]$, $\Theta(t)=\Theta_{t_0,\beta}(t)$ as in (\ref{def:THETA}). In this way $\Theta_{t_0,\beta}'(t)=-\theta_\beta(t-t_0)\leq 0$, and we have
\begin{equation}\label{viscous:limit3}
-\int_{0}^{T}\!\int_{-M-LT}^{M-LT} |u^\e(t,x)-u(t,x)| \,\zeta_M(|x| + L t) \theta_\beta(t-t_0)  dx\,dt + \e \ C \geq 0.
 \end{equation}
But $\zeta_M(|x| + L t)=1$ if $x\in (-M-LT,M-LT)$, and taking the limit $\beta\to 0^+$ in (\ref{viscous:limit3}) 
we find for all $t_0\in[0,T]$  (again using the Lebesgue differentiability theorem) that 
\begin{equation}\label{l1:estim:1}
\int_{-M-LT}^{M-LT} |u^\e(t_0,x)-u(t_0,x)| \zeta_M(|x| + L t_0)  dx \leq \e \ C.
 \end{equation}
The estimate for $t_0=T$ is obtained by letting $t_0 \to T$ in (\ref{l1:estim:1}).
An taking the limit $\e\to 0^+$ shows that $u^\e \to u$ in $C([0,T]; L^1_{loc}(\R))$ for all $T>0$.
Observe that $C$ depends on $\zeta''_M$, and for any $T$, and thus for any $M>LT$, we can choose $\zeta_M$ such that $\|\zeta''_M\|_\infty=2$, for instance.

{\it Proof of (b)}
We now leave the term $I_3$ unchanged, and we proceed as above for the rest of the argument, so that for all $t_0\in[0,T]$, instead of at (\ref{l1:estim:1}), we arrive at
\[
\begin{split}
& \int_{-M-L T}^{M-LT} |u^\e(t_0,x) - u(t_0,x)| \ dx \leq  C' w^{B}(\rho,\mu)\\
&+ \e \lim_{\beta\to 0}\int_{0}^{\infty}\!\int_{\R}\!\int_{0}^{\infty}\int_{\R} \Theta_{t_0,\beta}(t)\theta_\mu(s-t) |u^\e(t,x)-u(s,y)| 
\partial_x \overline{\DD^\alpha}[\omega_\rho(y-\cdot)\zeta_M(|\cdot|+Lt)](x) dx\,dt\,dy\,ds,
\end{split}
\]
then, observing that $\lim_{\beta\to 0}\Theta_{t_0,\beta}\leq \chi_{[0,t_0]}$, we obtain \begin{equation}\label{starting:point}
\begin{split}
& \int_{-M-L T}^{M-LT} |u^\e(t_0,x) - u(t_0,x)| \ dx \leq  C' w^{B}(\rho,\mu)\\
&+\e \int_{0}^{\infty}\!\int_{\R}\!\int_{0}^{t_0}\int_{\R} \theta_\mu(s-t) |u^\e(t,x)-u(s,y)| 
\partial_x \overline{\DD^\alpha}[\omega_\rho(y-\cdot)\zeta_M(|\cdot|+Lt)](x) dx\,dt\,dy\,ds.
\end{split}
\end{equation}

We now observe that the limit $\mu\to 0$ in (\ref{est1:w:B}) gives
\[
\begin{split}
&\lim_{\mu\to 0}(w^B(\rho,\mu))=
\frac{1}{\rho}\sup_{0<t<T}\left\{\int_\R \omega\left(\frac{z}{\rho}\right) \int_B|u(t,x)-u(t,z+x)|dx\,dz\right\}\\
&\leq
\frac{1}{\rho}\sup_{0<t<T}\left\{|u(t,\cdot)|_{BV}\int_{-\rho}^\rho \omega\left(\frac{z}{\rho}\right)|z|dz\right\}\\
&\leq\sup_{0<t<T}\left\{|u(t,\cdot)|_{BV}\int_{-\rho}^\rho \omega\left(\frac{z}{\rho}\right)dz\right\}\leq \rho\sup_{0<t<T}|u(t,\cdot)|_{BV}.
\end{split}
\]
Since $u_0 \in L^\infty(\R)\cap BV(\R)$, the entropy solution of (\ref{Problem3}) satisfies that $|u(t,\cdot)|_{BV} \leq |u_0|_{BV}$ (see e.g. \cite{Serre}), and we have
\[\lim_{\mu\to 0}(w^B(\rho,\mu))\leq \rho|u_0|_{BV},\]
Thus, the limit $\mu\to0$ of (\ref{starting:point}) is
\begin{equation}\label{starting:point2}
\begin{split}
&\quad \int_{-M-L T}^{M-L T} |u^\e(t_0,x) - u(t_0,x)| \ dx \leq C'\rho  |u_0|_{BV}  \\
&+\e \lim_{\mu\to 0} \int_{0}^{\infty}\,\int_{\R}\,\int_{0}^{t_0}\,\int_{\R} \theta_\mu(s-t) |u^\e(t,x)-u(s,y)| 
\partial_x \overline{\DD^\alpha}[\omega_\rho(y-\cdot)\zeta_M(|\cdot|+Lt)](x) dx\,dt\,dy\,ds.
\end{split}
\end{equation}

Let us get an estimate on the second term of the right-hand side of (\ref{starting:point2}). We integrate by parts with respect to $x$, then we estimate the absolute value:
\[
\begin{split}
\left| \int_{0}^{\infty}\,\int_{\R}\,\int_{0}^{t_0}\,\int_{\R} \theta_\mu(s-t) |u^\e(t,x)-u(s,y)| 
\partial_x \overline{\DD^\alpha}[\omega_\rho(y-\cdot)\zeta_M(|\cdot|+Lt)](x) dx \,dt\,dy\,ds \right|
\\
\leq \int_{0}^{\infty}\,\int_{\R}\,\int_{0}^{t_0}\, \int_{\R} \theta_\mu(s-t) \left|\partial_x u^\e(t,x)\right| 
\left|\overline{\DD^\alpha}[\omega_\rho(y-\cdot)\zeta_M(|\cdot|+Lt)](x)\right|  dx\,dt\,dy\,ds.
\end{split}
\]
The first two factors of the integrand do not depend on $y$, so we can integrate with respect to $y$ the remaining one. This one reads, applying Lemma~\ref{L1integrable} and Fubini's theorem, and for some arbitrary $r>0$:
\begin{equation}
\begin{split}
& \quad \int_{\R} \left|\overline{\DD^\alpha}[\omega_\rho(y-\cdot)\zeta_M(|\cdot|+Lt)](x)\right| \, dy \\
&\leq \int_{\R}\left\{\int_{0}^{r} |z|^{-\alpha} \left( |\omega_\rho'(y-x+z)| \,|\zeta_M(|x-z|+Lt)| + 
\omega_\rho(y-x+z)\,|\partial_z\zeta_M(|x-z|+Lt)|\right) \, dz\right.\\
&\quad + \left. \left| \int_{r}^{\infty} \alpha z^{-\alpha-1} \ 
\omega_\rho(y-x+z)\zeta_M(|x-z|+Lt) \ dz \right|\right\} \, dy \\
& \leq \Big(\|\omega_\rho'\|_1 \,\|\zeta_M\|_\infty +  \|\omega_\rho\|_1 \, \|\zeta_M'\|_\infty\Big) \frac{r^{-\alpha+1}}{1-\alpha} + \alpha^2  r^{-\alpha}\|\omega_\rho\|_1\, \|\zeta_M\|_\infty .
\end{split}
\end{equation}
We now use that $\|\omega_\rho\|_1 = 1$, $\|\omega_\rho'\|_1 \propto \frac{1}{\rho}$ and $\|\zeta_M\|_\infty=1$, $\|\zeta_M'\|_\infty < C$ for some $C$. This gives that there exist $C_1$, $C_2>0$ independent of $\e$, $\rho$, $\mu$ and $r$, such that
\begin{equation}\label{integrand:est}
\int_{\R} \left|\overline{\DD^\alpha}[\omega_\rho(y-\cdot)\varphi(\cdot,t)](x)\right|(x) \ dy 
\leq C_1 \left(\frac{r^{-\alpha+1}}{\rho}+ r^{-\alpha+1}\right) + C_2 r^{-\alpha}.
\end{equation}

By this last inequality, (\ref{starting:point2}) and the fact that $\int_{0}^{\infty}\theta_\mu(s-t)\ ds=1$, we find, for all $T>0$ and $M>L T$ there is $C_3>0$, such that
\begin{equation}\label{109:old}
\begin{split}
&\int_{-M-L T}^{M-L T} |u^\e(t_0,x) - u(t_0,x)| \, dx \leq C'\rho |u_0|_{BV} \\
&\quad+ \e  C_3\int_0^{t_0}|u^\e(t,\cdot)|_{BV}dt\, \left(\frac{r^{-\alpha+1}}{\rho}+ r^{-\alpha+1} +  r^{-\alpha}\right).
\end{split}
\end{equation}
On the other hand, by the $L^1$-contraction property, Theorem~\ref{L1:contraction}, and the translation invariance of the equation in (\ref{Problem1}), we have (see \cite{Serre})  
\[
|u^\e(t,\cdot)|_{BV} < |u_0|_{BV} , 
\]
which applied to (\ref{109:old}) gives
\begin{equation}\label{109}
  \int_{-M-L T}^{M-L T} |u^\e(t_0,x) - u(t_0,x)| \, dx \leq C'\rho |u_0|_{BV} 
  + \e  C_3\,t_0\,|u_0|_{BV} \, \left(\frac{r^{-\alpha+1}}{\rho}+ r^{-\alpha+1} + r^{-\alpha}\right).
\end{equation}

We then let $M\to \infty$ in (\ref{109}). We may take $\rho<1/2$, for instance, then there exists a constant $C>0$, such that, for all $t_0 \in [0,T]$,
\[
\|u^\e(t_0,\cdot) - u(t_0,\cdot)\|_{1} \leq C |u_0|_{BV} \left(\rho + 
\e \, t_0 \left( \frac{r^{-\alpha+1}}{\rho} + r^{-\alpha}\right)\right).
\]
Minimising the right-hand side of this inequality with respect to the variables $\rho$ and $ r$, we obtain that the minimum is attained at $\rho=(\e t_0)^{1/(\alpha+1)} (\alpha/1-\alpha)^{(1-\alpha)/(1+\alpha)}$ and $r=\rho\alpha/(1-\alpha)$, then for all $t_0 \in [0,T]$ we obtain (\ref{epsilontimevariation}) with $t_0$ replaced by $t$.
Taking the supremum over $t\in(0,T)$ we obtain the last assertion.
\end{proof}


\section{The travelling wave problem}\label{sec:TW}
In this section we study the vanishing viscosity limit for the travelling wave problem. In particular we consider 
solutions of (\ref{Problem1}) of the form  $u(t,x) = \phi(\xi) $ with $\xi = x - ct$ that connect different far-field values
 $\phi_-$, $\phi_+\in\mathbb{R}$. Then, the travelling wave problem reads:
\begin{equation}\label{TWProblem}
\begin{cases}
-c(\phi(\xi) - \phi_-) + f(\phi(\xi)) - f(\phi_-) = \DD^\alpha\phi(\xi), \\ 
\displaystyle\lim_{\xi \to -\infty}{\phi(\xi)}= \phi_- \ \text{ and } \displaystyle\lim_{\xi \to \infty}{\phi(\xi)}= \phi_+,
\end{cases}
\end{equation}
where after substitution of the new variables the equation has been integrated once using one of the far-field values. 
Moreover, integration over $\mathbb{R}$ gives that the wave speed $c$ must be given by the Rankine-Hugoniot condition
\begin{equation}\label{R-H}
c=\frac{f(\phi_+)- f(\phi_-)}{\phi_+ - \phi_-}>0.
\end{equation}

It is convenient to introduce the following notation for the left-hand side of the equation in (\ref{TWProblem})
\begin{equation}\label{h:def}
h(\phi):=-c(\phi(\xi) - \phi_-) + f(\phi(\xi)) - f(\phi_-).
\end{equation}
We further assume that $f$ is a convex function and that $\phi_->\phi_+$, then $h(\phi_-)=h(\phi_+)=0$ and 
\[
h'(\phi_-)>0 \quad \mbox{and} \quad h'(\phi_+)<0.
\] 
Existence of this problem has been established in \cite{AHS}. In particular, the authors obtain, under the more general assumption that the flux function $f$ 
is genuinely nonlinear (see also \cite{AHS2}), the following result: 
\begin{theorem}[Achleitner, Hittmeir, Schmeiser \cite{AHS}]\label{TWexis}
There exists a solution $\phi\in C^2_b(\mathbb{R})$ of (\ref{TWProblem}) such that
\[
\phi_+ \leq  \phi(\xi)\leq \phi_- \quad \mbox{for all}  \quad \xi\in\mathbb{R}
\]
and $\phi'(\xi)<0$, that is unique (up to a shift in $\xi$) among all $\phi\in \phi_-+ H^3(-\infty,0) \cap C_b^3(-\infty,0)$.
\end{theorem}

This theorem in \cite{AHS} appears with different notation and divided in a series of results that are proved step by step. 
Also, their results give less regularity than in the version above, it is, however, straightforward to show higher regularity of the solutions, see \cite{ACH} and \cite{CA}.

In this section we prove the following vanishing viscosity result:
\begin{theorem}\label{TW:vanish-vis}
If $\phi_\e$ is a solution of 
\begin{equation}\label{TWepsilon}
\begin{cases}
-c(\phi_\e(\xi) - \phi_-) + f(\phi_\e(\xi)) - f(\phi_-) = \e \DD^\alpha\phi_\e(\xi), \\ 
\displaystyle\lim_{\xi \to -\infty}{\phi_\e(\xi)}= \phi_- \ \text{ and } \displaystyle\lim_{\xi \to \infty}{\phi_\e(\xi)}= \phi_+,
\end{cases}
\end{equation}
then $\phi_\e \to \phi_0$ as $\e \to 0$ pointwise in $\R$, where 
\[
\phi_0(\xi)=
\begin{cases}
 \phi_- \quad \mbox{if}\quad \xi<\xi_0\\
 \phi_+ \quad \mbox{if}\quad \xi>\xi_0
\end{cases}
\]
for some $\xi_0$. Moreover,
\begin{equation}\label{exp:bh}
|\phi_\e(\xi) - \phi_-| = O\left(\exp(\lambda_\e\xi)\right) \ \text{ when } \xi \to -\infty, 
\ \text{ with } \lambda_\e=\left(\frac{h'(\phi_-)}{\e}\right)^{\frac{1}{\alpha}}>0
\end{equation}
and
\begin{equation}\label{alg:bh}
|\phi_\e(\xi) - \phi_+| = O\left(\frac{\e}{\xi^\alpha}\right) \ \text{ when } \xi \to +\infty.
\end{equation}
\end{theorem}

\begin{proof}
We observe that the change of variable $\phi(\xi)= \phi(\frac{\xi'}{\e^{1/\alpha}}) = \phi_\e(\xi')$ 
transforms problem (\ref{TWProblem}) into problem (\ref{TWepsilon}), so we can apply the existence result to (\ref{TWepsilon}) 
with the same conclusion, by simply adding the $\e$ dependency.
Then the pointwise limit follows from (\ref{exp:bh}) and (\ref{alg:bh}).

For the rest of the proof we take $\e=1$ without loss of generality, by the rescaling specified above. The behaviour (\ref{exp:bh}) of the travelling wave solutions for $\xi$ very negative is done in Lemma~2 of \cite{AHS}. This is in fact the starting point of the existence proof.

It remains to prove (\ref{alg:bh}). In this case, we already have the existence of solutions, and we can take $\phi$ as a known function and focus 
in the terms that involve very large $\xi$. We can rewrite the equation as follows, for some $\xi_\infty\gg 1$,
\begin{equation}
h(\phi(\xi)) = g(\xi) + \DD^\alpha_{\xi_\infty}[\phi](\xi),
\end{equation}
where we use the notation
\[
\DD^\alpha_{\xi_\infty}[\phi](\xi) := d_\alpha \int_{\xi_\infty}^{\xi}{\frac{f'(y)}{(\xi - y)^\alpha} dy},
\]
which is, up to a shift, a classical Caputo derivative, and the function
\[
g(\xi) := d_\alpha \int_{-\infty}^{\xi_\infty}{\frac{\phi'(y)}{(\xi-y)^\alpha} dy} \leq 0,
\]
(here we use that $\phi$ is decreasing).
We can now solve the equation implicitly, by the corresponding variation of constants formula, that is derived by using Laplace transform as it is done in \cite{GM2}. 
Namely, we introduce the new dependent variable $W= \overline{\phi} - \phi_+$, and a new independent one, $z=\xi-\xi_\infty$, so that $W$ satisfies
\begin{equation}\label{W:eq}
\DD^\alpha_0W(z) = h'(\phi_+)W(z)  + R(\phi(z+\xi^\infty),\phi_+) - g(z+\xi^\infty),
\end{equation}
where
\begin{equation}\label{W:eq:rest}
  R(\phi,\phi_+) = h(\phi) - h(\phi_+) - h'(\phi_+)(\phi - \phi_+)
  \geq 0.
\end{equation}
For the last inequality we use the convexity of $f$ in the interval $(\phi_+,\phi_-)$. 
We observe, that since $\phi$ is uniformly bounded and regular with bounded derivatives, there exists a constant $C>0$ such that
\begin{equation}\label{bd:R}
 |R(\phi(z+\xi^\infty),\phi_+)| = R(\phi(z+\xi^\infty),\phi_+) \leq C  \left( W(z)\right)^2, \quad z\geq 0. 
\end{equation}
If $W$ solves (\ref{W:eq})-(\ref{W:eq:rest}) 
with a given initial conditions $W(0)$, then it also satisfies
\begin{equation}\label{W:eq:impli}
\begin{split}
W(z)& = W(0)v(z) + \frac{1}{h'(\phi_+)} \int_{0}^{z}v'(y) Q(z-y) \ dy \\
 &\mbox{with} \quad Q(z)= R(\phi(z+\xi^\infty),\phi_+) - g(z+\xi^\infty)\geq 0,
\end{split}
\end{equation}
where 
\begin{equation}\label{W:eq:funda}
v(z)=\frac{1}{2\pi i}\int_{-\infty i + \sigma}^{+\infty i + \sigma}{e^{s z} \frac{s^{\alpha-1}}{s^{\alpha}- h'(\phi_+)} \ ds}
\quad \mbox{with} \quad \sigma\geq 1.
\end{equation}
We recall that $v$ is a positive decreasing function such that $\lim_{z\to 0^+}v(z)=1$, $\lim_{z\to\infty} v(z)=0$, 
$\lim_{z\to 0^+} v'(z)=-\infty$ and $\lim_{z\to\infty} v'(z)=0$, with the behaviours
\begin{equation}\label{v:infty}
v(z)\sim  \frac{C}{z^\alpha} \quad\mbox{as}\quad z\to \infty
\end{equation}
for some positive constant $C$ and 
\begin{equation}\label{v:prime:infty}
v'(z) \sim \frac{h'(\phi_+)}{\Gamma(\alpha)}z^{\alpha-1} \quad \mbox{as}\quad z\to 0^+
\end{equation}
(see \cite{GM2}).

We notice that the second term on the right-hand side of (\ref{W:eq:impli}) is non-negative, since $\phi$ is a decreasing 
function and $h'(\phi_+)<0$. Then, also applying (\ref{bd:R}), we obtain 
\begin{equation}\label{W:int:bd}
 W(0)\,v(z) \leq W(z)\leq W(0)\,v(z) + C_1 \int_{0}^{z}(-v'(y)) W(z-y)^2 \, dy + C_2 \int_{0}^{z}v'(y) \,g(z+\xi^\infty-y) \, dy
\end{equation}
with
\[
|g(z+\xi^\infty-y)|=  d_\alpha \int_{-\infty}^{0}\frac{(-W'(r))}{(z-y-r)^\alpha}\,dr,
\]
for some positive constants $C_1$ and $C_2$.

Let us first get an estimate on the last term of (\ref{W:int:bd}). We take $M>0$ large enough such that $|v'(z)|\leq C/z^{1+\alpha}$ for all 
$z\geq M$. Then, we split the integral of this term as follows:
\begin{equation}\label{bdd:I}
\begin{split}
I:= & \int_{0}^{M}|v'(y)| |g(z+\xi^\infty-y)|\,  dy + \int_{M}^z|v'(y)| |g(z+\xi^\infty-y)|\, dy\\
=& d_\alpha \int_{0}^{M}|v'(y)| \int_{-\infty}^{0}\frac{(-W'(r))}{(z-y-r)^\alpha}\,dr\, dy + d_\alpha\int_{M}^z|v'(y)|\int_{-\infty}^{0}\frac{(-W'(r))}{(z-y-r)^\alpha}\,dr\, dy.
\end{split}
\end{equation}

We notice that, by Lemma~\ref{equiv:repre} (\ref{integralformula2}) and the fact that $W$ is decreasing, we can write:
\begin{equation}\label{alternative:I}
  \int_{-\infty}^{0}\frac{(-W'(r))}{(z-y-r)^\alpha} \,dr = \alpha\int_{-\infty}^{0} \frac{W(r)-W(z-y)}{(z-y-r)^{\alpha+1}}\,dr + \frac{W(z-y)-W(0)}{ (z-y)^\alpha}.
\end{equation}
Then, we have two estimates, one deduced directly from the integral on the left-hand side of (\ref{alternative:I}),
\begin{equation}\label{int:r:bd}
\begin{split}
\int_{-\infty}^{0}\frac{|W'(r)|}{(z-y-r)^\alpha} \,dr &\leq \int_{-1}^{0}{\frac{C_1}{(z-y-r)^\alpha} \, dr} +
 \int_{-\infty}^{-1}{\frac{C_2 e^{\lambda r}}{(z-y+1)^\alpha} \, dr } \leq \frac{ C}{ 1 + (z-y)^\alpha}, 
\end{split}
\end{equation}
and another that can be deduced from the right-hand side of (\ref{alternative:I}),
\begin{equation}\label{int2:r:bd}
\begin{split}
 \int_{-\infty}^{0} \frac{|W(s)-W(z-y)|}{(z-y-s)^{\alpha+1}}\,ds &\leq 
\frac{ C_1}{(z-y)^{\alpha+1}}+ C_2\frac{W(z-y)}{(z-y)^\alpha}
\end{split}
\end{equation}
since $W$ is non-negative.

Then, for the first integral term in (\ref{bdd:I}), using (\ref{int2:r:bd}) and that $W$ is decreasing and non-negative, that
\begin{equation}\label{bdd:zero2M}
\begin{split}
  \int_{0}^{M}|v'(y)| \int_{-\infty}^{0}\frac{(-W'(r))}{(z-y-r)^\alpha}\,dr\, dy &\\
  \leq 
\int_{0}^{M} C y^{\alpha-1} \left( C_1 \frac{W(z-y)}{(z-y)^\alpha} + \frac{ C_2}{(z-y)^{\alpha+1}}\right)\,dy &\\\leq 
C\left( \frac{ M^\alpha W(z-M)}{(z-M)^\alpha} +  \frac{ M^\alpha}{(z-M)^{\alpha+1}} \right).&
\end{split}
\end{equation}
Now we take $M(z)$ as follows:
\[
M(z)=\sigma z \quad \mbox{for} \quad \sigma\in (0,1),
\]
where we will later take $\sigma$ as small as necessary. This gives, for some positive constant $C$ independent of $\sigma$,
\begin{equation}\label{est1}
\begin{split}
\int_{0}^{\sigma z}|v'(y)| \int_{-\infty}^{0}\frac{(-W'(r))}{(z-y-r)^\alpha}\,dr\, dy \leq 
C\left( \sigma^\alpha W(z-M) +  \frac{1}{(z-M)} \right).
\end{split}
\end{equation}
Observe that  $z-M(z)= (1-\sigma)z$.

For the second term of (\ref{bdd:I}) (an integral over $(M(z),z)$) we get,
\begin{equation}\label{est2}
\begin{split}
\int_{\sigma z}^{z}|v'(y)| \int_{-\infty}^0 \frac{-W'(r)}{(z-y-r)^{\alpha}}\,dr\, dy
\leq C\int_{\sigma z}^{z} \frac{1}{y^{1+\alpha}}  \frac{1}{1+(z-y)^{\alpha}} \,dy \\
\leq 
C\int_{\sigma z}^{z}  \frac{1}{y^{1+\alpha}} \,dy
 \leq C \frac{1}{z^{\alpha}}.
\end{split}
\end{equation}
Combining (\ref{est1}) and (\ref{est2}), we obtain
\begin{equation}\label{est:I}
I\leq C\left( \sigma^\alpha W(z(1-\sigma)) + \frac{1}{z^\alpha}+  \frac{1}{z (1-\sigma)} \right).
\end{equation}

It remains to get an estimate on the second term of (\ref{W:int:bd}). We proceed similarly, by splitting the integral:
\begin{equation}\label{est:I:prime}
  \begin{split}
    I':=\int_0^{M'} (-v'(z))W(z-y)^2dy+\int_{M'}^z (-v'(z))W(z-y)^2dy\\
    \leq C\left( (M')^\alpha W(z-M')^2 - \frac{W(0)^2}{z^\alpha} + \frac{W(0)^2}{(M')^\alpha}\right).
\end{split}
  \end{equation}
We then take $M'$ such that $M'(z)^\alpha=\delta /W(z)$ with $\delta\in(0,1)$ for $z$ large. For each $z$ fixed and large, we can take $\sigma$ and
 $\delta$ small enough, such that, the estimate (\ref{est:I}) together with (\ref{est:I:prime}) imply that 
\[
\frac{1}{C} W(z) \leq W(z) \left(1- c\sigma^\alpha -c' \delta \left(1 + \frac{W(0)^2}{\delta^2}\right)\right)\leq C\frac{1}{z^\alpha} .
\]
for some $C>1$. This is possible by taking $W(0)$ as small as necessary once $\sigma$ and $\delta$ are fixed. Recall that $W(0)=\phi(\xi_\infty)-\phi_+$,
 and $\xi_\infty$ can be chosen sufficiently large so that $W(0)$ is arbitrary small. Also all other constants are independent of $\xi_\infty$, $\sigma$ 
and $\delta$, also $\delta$ and $\sigma$ are independent of each other·

On the other hand, applying the right hand side inequality of (\ref{W:int:bd}) and the behaviour of $v(z)$ for $z$ large (\ref{v:infty}), 
we obtain that there exists a constant, depending on $\xi_\infty$, such that
\[
\frac{1}{C_\infty}\frac{1}{z^\alpha} \leq W(z) \leq C_\infty \frac{1}{z^\alpha} \quad  \mbox{as}\quad z\to\infty.
\]
This finishes the proof.
\end{proof}

\section{Generalisation to regularizations by general Riesz-Feller operators}\label{sec:rieszfeller}
As anticipated in the Introduction, in this section, we explain how our results of sections \ref{sec:pre}, \ref{sec:entropy:l1} and \ref{sec:limit} also hold for (\ref{Problem:general}), where the non-local regularization is given by the operator (\ref{symbol1:general})-(\ref{symbol2:general}).


Existence and regularity results are obtained similarly by defining mild solutions as in Definition~\ref{solution} with the kernel
\[
K_{\gamma}^\beta(t,x) := \FF^{-1}\left( e^{t\psi_\gamma^\beta(\cdot)} \right)(x)  
\]
instead of $K$, and deriving the properties of $K_{\gamma}^\beta$, that hold as in Proposition~\ref{K:prop} (for the proofs we refer to \cite{AK2} Lemma~2.1). One can also show all the other results of Section~\ref{sec:pre} in a similar way. In particular, the obvious extensions of Theorem~\ref{local:existence} and Proposition~\ref{time:reg} hold. 

Concerning the results of sections \ref{sec:entropy:l1} and \ref{sec:limit}, we observe that generalisations can be obtained with minimal effort. This is because equivalent integral representations of the operators (\ref{symbol1:general})-(\ref{symbol2:general}) acting on $C^2$ functions can be obtained, see e.g. \cite{Sato} and \cite{AK2}. Thus we need to generalise the weak entropy inequalities, for both the $L^1$ contraction property and the zero viscosity limit. The key integral representation of this type of operator is the following:

\begin{proposition}[\cite{AK2} Theorem 2.4 and, e.g., \cite{Sato}]\label{repre:general} If $1< \beta<2$ and $|\gamma|\leq \min\{\beta,2-\beta\}$, then for all $v\in S(\R)$ and $x\in\R$
  \begin{equation}\label{RieszFeller:int}
  D^\beta_\gamma [v](x)  =
  c_\gamma^1 \int_0^\infty \frac{v(x+z) - v(x) - v'(x) z}{z^{1+\beta}} \,dz
 + c_\gamma^2\int_0^\infty \frac{v(x-z) - v(x) + v'(x) z}{z^{1+\beta}} \,dz
  \end{equation}
  for some constants $c_\gamma^1$, $c_\gamma^2 \geq 0$ with $c_\gamma^1+c_\gamma^2>0$. This singular integral representation is well-defined for $C^2_b$ functions and is such that
  $D^\beta_\gamma$ maps $C^2_b(\R)$ to $C_b(\R)$ and is a bounded operator.

  Moreover, for $ v \in C^2_b(\R)$ one has the estimate
  \[
  \sup_{x\in\R}|D^\beta_\gamma [v](x)|\leq \frac{1}{2}(c_\gamma^1 + c_\gamma^2) \|v''\|_{C_b(\R)}
  \frac{M^{2 - \beta}}{2 - \beta}+ 2(c_\gamma^1 + c_\gamma^2)\|v'\|_{C_b(\R)} \frac{M^{1 - \beta}}{\beta - 1}<\infty
  \]
  for some positive constant $M$ and $c_\gamma^1$ and $c_\gamma^2$ as above.
\end{proposition}

Now, the maximum principle as stated in Lemma~\ref{negativity} holds for $D^\beta_\gamma$ (the proof is similar by first using the representation (\ref{RieszFeller:int})) and thus global existence for (\ref{Problem:general}) is proved similarly. In particular, the analogous of Proposition~\ref{global:existence} is satisfied for mild solutions of (\ref{Problem:general}).

We note that the last estimate in Proposition~\ref{repre:general} above follows also by proving the equivalent representation for regular functions of the integral terms in (\ref{RieszFeller:int}), that is:
\[
\int_0^\infty \frac{v(x+z) - v(x) - v'(x) z}{z^{1+\beta}} \,dz=\frac{1}{\beta(\beta-1)} \partial_x\int_{-\infty}^0\frac{v'(x+r)}{|r|^{\beta-1}}dr,
\]
and
\[
\int_0^\infty \frac{v(x-z) - v(x) + v'(x) z}{z^{1+\beta}} \,dz=\frac{1}{\beta(\beta-1)} \partial_x\int^{\infty}_0\frac{v'(x+r)}{|r|^{\beta-1}}\,dr.
\]
These identities follow from the steps in the proof of Lemma~\ref{equiv:repre}. Then for $C^2$ functions we have
\begin{equation}\label{rel:operators}
D^\beta_\gamma[ v](x)= \frac{1}{d_{\beta+1}} \left(c_\gamma^1\partial_x\DD^{\beta-1}[v](x)
+ c_\gamma^2 \partial_x\overline{\DD^{\beta-1}}[v](x)\right),
\end{equation}
where $\partial_x\overline{\DD^{\beta-1}}$ and $\overline{\DD^{\beta-1}}$ are defined and characterised in Lemma~\ref{integrationbyparts} with $\beta-1=\alpha$.

With the representation of Proposition~\ref{repre:general} and that in (\ref{rel:operators}) we obtain the following:
\begin{proposition}\label{key:general} Let $u\in C_b^2((0,\infty)\times \R)$, then
  \begin{itemize}
  \item[(i)] For all $\eta\in C^2(\R)$ convex and $\varphi\in C_b^1(\R)$, it holds
    \[
    D^\beta_\gamma[\eta(\varphi)](x) \geq \eta'(\varphi) D^\beta_\gamma[\varphi](x).
    \]
    
  \item[(ii)] For all $\varphi\in C_c^\infty(\R)$, then
    \begin{equation*}
\int_{0}^{\infty}\int_{\R} \varphi(x)\,D^\beta_\gamma [u(t,\cdot)](x) \, dx\, dt =
 \int_{0}^{\infty}\int_{\R}  \overline{D^\beta_\gamma}[ \varphi](x)\, u(t,x) \, dx\,dt,
    \end{equation*}
    where
    \[
    \overline{D^\beta_\gamma}[g](x)=\frac{1}{d_{\beta+1}} \left(c_\gamma^1\partial_x\overline{\DD^{\beta-1}}[v](x)
+ c_\gamma^2 \partial_x\DD^{\beta-1}[v](x)\right).
    \]
  \item[(iii)] If $u(t,\cdot)-v(t, \cdot)\in L^1(\R)\cap C_b^2(\R)$, then
    \[
    \sign(u(t,x)-v(s,y))\,\left(\partial_x\overline{\DD^{\beta-1}}[u(t,\cdot)](x)-\partial_y\overline{\DD^{\beta-1}}[v(s,\cdot)](y)\right) \leq
    \overline{\DD}^{\beta}_{x,y}[|u(t,\cdot)-v(s,\cdot)|](x,y),
\]
where the operator $\overline{\DD}_{x,y}^{\beta}={}^0\overline{\DD}_{x,y}^{\beta}$, and ${}^r\overline{\DD}_{x,y}^{\beta}$ is defined in (\ref{def:new:Dadjoint}).
  \end{itemize}
\end{proposition}

\begin{proof}
  {\it (i)} is proved by using the integral representation and the convexity of $\eta$ (as for (\ref{convex:multi})).

  {\it (ii)} This follows from (\ref{rel:operators}) and Lemma~\ref{integrationbyparts}.

  {\it (iii)} This estimate follows as in (\ref{L1:key:est}). 
\end{proof}

With {\it (i)} and {\it (ii)} of Proposition~\ref{key:general}, the following generalisation of Theorem~\ref{entropy:smooth} holds, the proof being analogous:
\begin{theorem}\label{entropy:smooth:general} Given $\eta\in C(\R)$ convex and $u\in C^2_b((0,\infty)\times\R)$ a solution of (\ref{Problem:general}), then for all 
$\varphi \in C_c^\infty((0,\infty)\times \R)$
\begin{equation}\label{entropy:general}
  \int_{0}^{\infty}\int_{\R} 
\Big( 
\eta(u(t,x)) \partial_t\varphi(t,x) + q(u(t,x)) \partial_x \varphi(t,x) 
+ \eta(u(t,x)) \overline{D^\beta_\gamma}[\varphi(t,\cdot)](x)
\Big) \, dx\, dt\geq 0,
  \end{equation}
\end{theorem}

Theorem~\ref{L1:contraction} holds unchanged for (\ref{Problem:general}). In order to prove this, we follow the same steps, the main difference is that we have one more term in the nonlocal operator, which can be written as (\ref{rel:operators}). This is no substantial difference, since the crucial estimate, that allows to show the pertinent entropy inequality, is (\ref{L1:key:est}). But with this same one and (iii) of Proposition~\ref{key:general} above, we obtain
\[
\sign(u(t,x)-v(s,y))\,
\left(D^{\beta}_\gamma[u(t,\cdot)](x)-D^{\beta}_\gamma[v(s,\cdot)](y)\right) \leq
    D^{\beta}_{\gamma,x,y}[|u(t,\cdot)-v(s,\cdot)|](x,y),
\]
where the operator $D^{\beta}_{\gamma,x,y}$ is defined by means of
\[
D^{\beta}_{\gamma,x,y}[g](x,y)=\frac{1}{d_{\beta+1}} \left(c_\gamma^1\DD^{\beta}_{x,y}[v](x,y)
+ c_\gamma^2 \overline{\DD}^{\beta}_{x,y}[g](x,y)\right).
\]
This gives, with the same choice of test functions, the entropy inequality (\ref{3.7}) where the non-local operator in the last term of the integrand is replaced by
\[
\frac{1}{d_{\beta+1}}\left(c_\gamma^1\partial_{z'} \overline{\DD^{\beta-1}}[\varphi(r',\cdot)](z')
+ c_\gamma^2 \partial_{z'} \DD^{\beta-1}[\varphi(r',\cdot)](z')\right).
\]
The rest of the proof follows similarly, the treatment of the terms coming from this being analogous.

Finally, the zero viscosity limit results follow similarly with the aid of Theorem~\ref{entropy:smooth:general} above. The treatment of the nonlocal term being analogous, once the nonlocal operators are written using (\ref{rel:operators}) and {\it (ii)} of Proposition~\ref{key:general}. In this way, Theorem~\ref{vanishingregularization} holds unchanged for (\ref{Problem:general}).

{\bf Acknowledgements:} The authors greatly thank Franz Achleitner for helpful comments and a thorough review of the first version of the manuscript. The authors also acknowledge the financial support of the Spanish Government through the MICINNU project PGC2018-094522-B-I00, and of the Basque Government through the Research Group grant IT1247-19. Xuban Diez-Izagirre also acknowledges the support of the Basque Government through the doctoral grant PRE-2018-2-0013.

\bibliographystyle{plain}

\end{document}